\documentclass[10pt]{amsart}
\usepackage[english]{babel}
\usepackage[utf8x]{inputenc}
\usepackage[T1]{fontenc}

\usepackage{graphicx, amsmath, amssymb, amsthm, amscd,mathrsfs,}
\usepackage{tikz}
\usepackage{color}
\usepackage{epsf}
\usepackage{enumerate}

\usepackage{graphicx}
\usepackage[colorinlistoftodos]{todonotes}
\usepackage[colorlinks=true, allcolors=blue]{hyperref}

\newtheorem{theorem}{Theorem}[section]
\newtheorem{lemma}[theorem]{Lemma}
\newtheorem{remark}[theorem]{Remark}

\newtheorem{claim}[theorem]{Claim}
\newtheorem{definition}[theorem]{Definition}
\newtheorem{corollary}[theorem]{Corollary}
\newtheorem{question}{Question}

\title[Prime Ends Dynamics in Parametrised Families]{Prime Ends Dynamics in Parametrised Families of Rotational Attractors}
\begin{document}
\maketitle

\begin{center}Jan P. Boro\'nski\footnote{National Supercomputing Centre IT4Innovations, Division of the University of Ostrava, Institute for Research and Applications of Fuzzy Modeling, 30. dubna 22, 70103 Ostrava, Czech Republic
e-mail: jan.boronski@osu.cz.}, 
Jernej \v Cin\v c\footnote{National Supercomputing Centre IT4Innovations, Division of the University of Ostrava, Institute for Research and Applications of Fuzzy Modeling, 30. dubna 22, 70103 Ostrava, Czech Republic
	-- and --AGH University of Science and Technology, Faculty of Applied Mathematics, al. Mickiewicza 30, 30-059 Krak\'ow, Poland.
e-mail: jernej.cinc@osu.cz.}, and
Xiao-Chuan Liu\footnote{Instituto de Matem\'atica e Estat\'istica da Universidade de S\~ao Paulo, R. do Mat\~ao, 1010 - Vila Universitaria, S\~ao Paulo, Brasil.
e-mail: lxc1984@gmail.com.}
\end{center}

\begin{abstract}
We provide several new examples in dynamics on the $2$-sphere,	with the emphasis on better understanding the induced boundary dynamics of invariant domains in parametrized families.
First, motivated by a topological version of the Poincaré-Bendixson Theorem obtained recently by Koropecki and Passeggi, we show the existence of homeomorphisms of $\mathbb{S}^2$ with Lakes of Wada rotational attractors, with an arbitrarily large number of complementary domains, and with or without fixed points, that are arbitrarily close to the identity. This answers a question of Le Roux.
Second, from reduced Arnold's family we construct a parametrised family of Birkhoff-like cofrontier attractors, where at least for uncountably many choices of the parameters, two distinct irrational prime ends rotation numbers are induced from the two complementary domains. This example complements the resolution of Walker's Conjecture by Koropecki, Le Calvez and Nassiri from 2015.
Third, answering a question of Boyland, we show that there exists a non-transitive Birkhoff-like attracting cofrontier which is obtained from a BBM embedding of inverse limit of circles, such that the interior prime ends rotation number belongs to the interior of the rotation interval of the cofrontier dynamics. There exists another BBM embedding of the same attractor so that the two induced prime ends rotation numbers are exactly the two endpoints of the rotation interval.
\end{abstract}

\section{Introduction}
The prime ends rotation number induced by surface homeomorphisms restricted to an invariant disk
is one of the important invariants  
in the study of boundary dynamics. 
Parametrised families of dynamical systems can provide 
a clearer view of both the surface dynamics and the boundary dynamics in many situations. 
In this paper, our study serves as a contribution in this direction, 
  by providing new  examples in various interesting contexts.  
  Let us postpone to Subsection 2.2 for 
  several standard terminologies needed below.

We were
initially motivated by 
the following recent result, 
referred to as 
topological version of the Poincar\'e-Bendixson Theorem. 
\vspace{0.25cm}

\begin{quote}
\textbf{Theorem A. (Koropecki and Passeggi, \cite{Korop})}
\textit{ Let $\Gamma$ be 
a translation line for an orientation-preserving homeomorphism of 
$\mathbb{S}^2$. 
Then either the omega limit set 
$\omega(\Gamma)$ contains a fixed point, or
 $\Gamma$ 
is a topologically embedded line,
and its filled $\omega$-limit set $\widetilde{\omega}(\Gamma)$ is a rotational attractor,
 disjoint from $\Gamma$.}
 \end{quote}

 In connection with Theorem A, 
 we will construct a parametrised family of 
 Lakes of Wada continua, which are 
 the omega limit sets of translation lines, 
 and whose filled continua are rotational attractors.
  We will also study the corresponding 
 exterior prime ends rotation numbers. The approach we take is the so-called 
Brown-Barge-Martin (BBM) embeddings
of inverse limits of topological graphs
(see~\cite{BM} and~\cite{Brown}).
This tool has proven to play an important role in constructing  
surprising new examples for topological dynamical systems. 
A particularly useful extension of this 
method is provided by 
the parametrised version of BBM embedding, 
proved recently by Boyland, 
de Carvalho and Hall \cite{BCH}.  This generalized version makes it possible
 to know precisely how 
 in a family of maps, the rotation set changes. The same authors used this method as a tool 
 to find new rotation sets for torus homeomorphisms (see \cite{Boyland_New}) as well as to study prime ends of natural extensions of unimodal maps (see \cite{Boy}).
 However, for our purposes in this paper, some adaptations of this technique is necessary. 
 
\subsection{Statements of the Results}
   A {\em continuum} is a compact and connected metric space. A continuum $X$ in the sphere $ \mathbb{S}^2$ (or the plane) is said to be {\em $n$-separating} if $ \mathbb{S}^2\setminus X$ has $n$ connected components. A {\em Lakes of Wada continuum} is an $n$-separating continuum in $ \mathbb{S}^2$ 
  which is the common boundary of each of the $n>2$ components of its complement.  
  There are well known examples of attractors that are Lakes of Wada
  continua arising as projections of DA-attractors from the torus onto the sphere, as well as those constructed directly on the sphere by Plykin~\cite{Plykin}. 
  Other examples of Lakes of Wada continua,
  as well as their relations with physical phenomena,
  were given in \cite{KY} and bifurcations of basins of attraction from the view point of prime ends were studied for planar diffeomorphisms in \cite{NY}. Lakes of Wada property for trapping regions was studied in \cite{NYScience}. In the present paper the study of Lakes of Wada attractors is motivated by Theorem A,
 which in turn can be viewed as a particular generalization of the classical Poincar\'e-Bendixon Theorem.
  In view of Theorem A, during the conference \textit{Surfaces in Luminy}, 
held at 
 Centre International de Rencontres Math\'ematiques, October 3 - 7, 2016, 
 Fr\'ed\'eric Le Roux asked 
 if the boundary of the rotational attractor $\widetilde{\omega}(\Gamma)$ 
 could be a Lakes of Wada continuum. 
 Possible existence of such examples was 
 conjectured in~\cite{Korop}.
 Clearly, in order to construct such an example, it is necessary to 
 understand
 the corresponding exterior prime ends dynamics better. 
However, the literature suggests that the prime ends rotation numbers for Lakes of Wada attractors
 have not been systematically studied yet. Another feature is as follows. The planar attractors we obtain in 
 Theorem~\ref{parametrised} and Theorem~\ref{thm:noFP}  
 share the Lakes of Wada topology with the Plykin attractors. 
 However, our attractors arise for dynamical systems which are arbitrarily small 
 perturbations of identity, and they are not expansive (see Remark~\ref{nonexpansive} for details), contrasting with expansivity of the Plykin attractors resulting from the global stretching and bending on the sphere. Up to our knowledge, ours is the first example of such phenomena in close vicinity of the identity.
 
 Let $d_H$ denote the Hausdorff distance 
between two compact subsets of $\mathbb{S}^2$, and let $d_{C^0}$ denote the usual $C^0$ topology in the space of 
homeomorphisms of some metric space. 
The following is the main result 
of this paper.

\begin{theorem}\label{parametrised}
For each $n> 2$, there exists a family of homeomorphisms 
$\Phi_{\epsilon}: \mathbb{S}^2 \to  \mathbb{S}^2$, 
and a family of $\Phi_\epsilon$-invariant continua $K_\epsilon$, 
where $\epsilon\in (0,\epsilon_0]$ for some $\epsilon_0>0$,  such that:
\begin{enumerate}[(i)]
    \item $\underset{\delta\to 0}{\lim} d_H(K_{\epsilon},K_{\epsilon+\delta}) =0$, 
    for each $\epsilon\in(0,\epsilon_0]$.
    \item 
     the external prime ends rotation numbers, denoted by
      $\rho_{\text{ex}}(\Phi_{\epsilon},K_{\epsilon})$,  
        is a strictly increasing function of $\epsilon$. Moreover, 
    \begin{equation}
    \underset{\epsilon\to 0} {\lim} \rho_{\text{ex}}(\Phi_{\epsilon},K_{\epsilon})= 0.
    \end{equation}
      \item for any  $\epsilon\in (0,\epsilon_0]$, the continuum
       $K_\epsilon$ is an
      $n$-separating Lakes of Wada continuum, 
      which is also a rotational attractor.   
       \item $\underset{\epsilon\to 0}{\lim} d_{C^0}(\Phi_{\epsilon},id) = 0$.
    \item the topological entropy
    $h_{\text{top}} (\Phi_{\epsilon}|_{K_{\epsilon}})$ 
    is strictly decreasing as $\epsilon\to 0$, and 
    \begin{equation}
    \lim_{\epsilon\to 0}h_{\text{top}} ( \Phi_{\epsilon}|_{K_{\epsilon}} ) =0.
    \end{equation}
    \item for every $\epsilon\in(0,\epsilon_0]$,
    there exists a translation line 
    $\underline \Gamma\subset \mathbb{S}^2$ 
    so that $\omega(\underline \Gamma)=K_{\epsilon}$.
    The filled omega limit of $\underline \Gamma$,
    denoted 
    by $\widetilde{\omega}(\underline \Gamma)$,
    is a rotational attractor, which is  disjoint from $\underline \Gamma$. 
\end{enumerate}
\end{theorem}

The continua $K_{\epsilon}$ 
from Theorem~\ref{parametrised} have a $\Phi_{\epsilon}$-fixed point.
However, one can also obtain 
examples without fixed points, 
as we will show in the next theorem, 
thus providing examples for both cases stated in Theorem A. For the purpose of the discussion that follows Theorem~\ref{thm:noFP}, we state this next result on the plane. 

\begin{theorem}\label{thm:noFP}
For any $n\geq 2$, there exists a $2n$-separating 
Lakes of Wada rotational attractor $K$, and a homeomorphism 
$\Phi:\mathbb{R}^2\to\mathbb{R}^2$ such that $\Phi|_{K}$ is fixed-point 
free. There exists a translation line $\Gamma\subset \mathbb{R}^2$ so that $\omega(\Gamma)=K$ and $\widetilde{\omega}(\Gamma)$ is a rotational attractor disjoint from $\Gamma$.
\end{theorem}

We are in an unfortunate position to have to point out an error in a research announcement \cite{Brechner}, where in the Introduction it is stated, that corollaries of methods proposed in \cite{Brechner} imply the following: every homeomorphism of the Lakes of Wada continuum, that is extendable to the plane, must have a fixed point in the composant accessible from the unbounded complementary domain. Theorem~\ref{thm:noFP} implies that there exist homeomorphisms of the Lakes of Wada continua that are extendable to the plane where no composants contain any fixed points of the homeomorphism. Therefore, it follows that composants which are accessible from the unbounded complementary domain have no fixed point of the homeomorphism as well, implying an error in the results announced in \cite{Brechner}.

We now turn our attention to another class of examples. 
Suppose some orientation-preserving homeomorphism on $ \mathbb{S}^2$  admits an invariant separating continuum, which is 
a common boundary of two complementary domains (such a continuum is called a {\em cofrontier}). 
This provides another interesting context, where 
prime ends rotation number comes into play.
A conjecture of 
Walker~\cite{Walker} stated that
it is not possible for
the induced dynamics on 
the two prime ends circles to be
conjugate to two rigid rotations with different irrational rotation numbers. 
Recently, this conjecture 
was proved by 
Koropecki, Le Calvez and Nassiri (see Theorem F of \cite{KLCN}). 
 However, in the spirit of this conjecture, it is still interesting to find out if it is possible that two primes ends rotation numbers are two different irrationals. The following result gives an affirmative answer to this question through a family of examples. To state the result, assuming that we are working with invariant cofrontiers, in what follows let 
 $\rho_{\text{in}}$ 
 denote the prime ends rotation number of the induced dynamics on 
 the interior circle of prime ends, defined analogously as in the exterior prime ends rotation number. 
  A \emph{Birkhoff-like attractor} 
 for an orientation-preserving homeomorphism $g:\Bbb S^2 \to \Bbb S^2$ 
 is an attractor 
 $\Lambda\subset  \mathbb{S}^2$ 
 such that the rotation set $\rho(g|\Lambda)$ is a nondegenerate interval. 

\begin{theorem}\label{pseudo-circles}
There exists a parametrised 
family of homeomorphisms $\{\Phi_t\}_{t\in I}: \mathbb{S}^2\to  \mathbb{S}^2$, which is parametrised by 
a closed interval $I$, such that
each $\Phi_t$
preserves a
cofrontier attractor $K_t\subset  \mathbb{S}^2$, with the following properties.
\begin{enumerate}[(i)]
    \item $K_t$ is a Birkhoff-like attractor for $\Phi_t$, for each $t\in I$,
    \item  for each $t \in I$ and for each $\epsilon>0$,
     there exists a sphere homeomorphism $\Phi'_t$
    which preserves a 
    Birkhoff-like attractor $K'_t$ homeomorphic to the pseudo-circle, 
    such that $d_H(K_t, K'_t)<\epsilon$, and $d_{C^0}(\Phi_t,\Phi'_t)<\epsilon$
    \item there are uncountably many choices 
    of the parameter 
    $t \in I$, 
    such that the lifted interior and exterior prime ends rotation numbers 
    $\rho_{\text{ex}}(\widetilde \Phi_t,K_t)=
    \eta_t\neq\eta'_t=
    \rho_{\text{in}}(\widetilde \Phi_t,K_t)$, 
    with $\eta_t,\eta'_t\notin \mathbb{Q}$.  
\end{enumerate}
\end{theorem}
Note that Theorem F of \cite{KLCN} 
implies that when prime ends rotation numbers from Theorem~\ref{pseudo-circles} are irrational, the induced dynamics on two circles of prime ends are Denjoy homeomorphisms.
In view of item (iii) Theorem~\ref{pseudo-circles}, 
we can ensure that 
each cofrontier attractor $K_t$ is indecomposable, 
but it is yet to be determined whether such $K_t$ can all be homeomorphic to Bing's pseudo-circle \cite{Bing}. 
Note that the homeomorphism group of the pseudo-circle does not contain any non-degenerate continuum \cite{Lewis83}, 
so in a potential family the embeddings would need to change continuously, in such a way as to produce continuosly varying prime ends rotation numbers, but the homeomorphisms could not change in such a way. 

\begin{question}
	Can the
	interior and exterior prime ends rotation numbers
	be distinct irrational numbers for a cofrontier that is a pseudo-circle?
\end{question}
\begin{question}
	Does there exist a parameterized 
	family of  homeomorphisms 
	$\{ \Phi_t:\mathbb{S}^2\to\mathbb{S}^2\}_{t\in I}$ such that 
	for any $t\in I$, 
	$\Phi_t$ preserves 
	a cofrontier attractor $K_t$, 
	which 
	is 
	homeomorphic to the pseudo-circle? 
\end{question}

Another problem that can be dealt with proper embeddings of the inverse limit dynamics 
is as follows. 
During the Workshop on Dynamical Systems and Continuum Theory, held at 
University of Vienna, 
June 29 - July 3, 2015,
Philip Boyland raised the following question: 
\begin{quote}
Suppose that $\widetilde f$ is a lift of a circle endomorphism $f : \mathbb{S}^1 \to \mathbb{S}^1$ of degree $1$, 
and assume that its rotation interval $\rho(\widetilde f)=[a,b]$.
Then is it true that,
the exterior and interior 
prime ends rotation numbers of the unwrapping of 
$f$ via inverse limit method
 in the annulus are exactly $a$ and $b$, respectively?
\end{quote}
For Birkhoff attractors, 
it is always true that the prime ends rotation numbers are the endpoints of 
the rotation interval (see \cite{Birkhoff_attractor_Patrice}).
Related to this, Boyland also asked if complicated inverse limit spaces (of the interval or the circle)
can be embedded in $ \mathbb{S}^2$
in multiple ways, which was already answered for unimodal inverse limits in \cite{A.Anuvsic2017}. Here we answer these two questions by showing the following result.
\begin{theorem}\label{NoEndPoint}
There exists an orientation-preserving and homology-preserving 
annulus homeomorphism $\Phi:\mathbb{A}\to\mathbb{A}$,
 with an attracting non-transitive Birkhoff-like cofrontier $K$,
  such that the lifted interior prime ends rotation number 
  $\rho_{\text{in}}(\widetilde \Phi,K)$ 
  is contained in the interior of 
  the interval  
  $\rho(\widetilde \Phi)$. 
  There exists another embedding of the pair $(K,\Phi\big |_K)$ as an attractor in $\mathbb{A}$, so that the two induced lifted  prime ends rotation numbers
  are exactly the 
  two endpoints of the rotation interval 
  of the corresponding natural extension.
\end{theorem}
The paper is organized as follows. In Section~\ref{preliminaries}, we present
some standard notation and state some related results that we will use in the paper, 
in particular a parametrised version of the BBM method. 
In Section~\ref{endo_rotation_lemmas}, 
we will prove a lemma about circle endomorphisms needed for future use. 
In Section~\ref{lakes_wada}, we give a proof of Theorem~\ref{parametrised}, which 
is the core of this paper. In Section~\ref{with_no_fixedpoints} we show Theorem~\ref{thm:noFP}. 
Finally, in Section~\ref{applications}, 
we prove Theorem~\ref{pseudo-circles} and Theorem~\ref{NoEndPoint}.

\section{Preliminaries}\label{preliminaries}
\subsection{Inverse Limit Spaces and the BBM Embedding}
Let $X$ be a metric space. Two homeomorphisms 
$f,g:X\to X$ are called \emph{topologically conjugate}, 
if there exists a homeomorphism 
$h:X\to X$ so that $h\circ g\circ h^{-1}= f$.
Denote by $\mathcal{C}(X,X)$ (respectively, $\mathcal{H}(X,X)$) 
 the set of all continuous mappings 
 on a metric space $X$ (respectively, the set of all homeomorphisms of $X$).

Let $\mathbb{N}_{0}$ denote the set of non-negative integers $\{0,1,\cdots\}$.
Our main tool for constructing the examples are 
\emph{inverse limit spaces}. For $f \in \mathcal{C}(X,X)$,
we denote
\begin{equation}
\underleftarrow{\lim} \{X,f\} 
:=
\{\big(\ldots,x_{-1},x_0 \big) \in X^{-\mathbb{N}_0} \big|  
x_i\in X, x_i=f(x_{i-1}), \text{ for any }i \leq 0\}.
\end{equation}
We equip $\underleftarrow{\lim} \{X,f\} $ with the subspace 
metric induced from the 
\emph{product metric} in $X^{-\mathbb{N}_0}$,
where $f$ is called the {\em bonding map}. 
The inverse limit space $\underleftarrow{\lim} \{X,f\}$
also comes with a natural homeomorphism, 
called the \emph{natural extension} of $f$, or the 
\emph{shift homeomorphism}
$\sigma_{f}:\underleftarrow{\lim}\{X, f\}
\to \underleftarrow{\lim}\{X, f\}$, 
defined as follows. 
For any $\underbar x:= \big(\ldots, x_{-2},x_{-1},x_0 \big)\in \underleftarrow{\lim}\{X,f\}$,
\begin{equation}
\sigma_f(\underbar x):= \big(\ldots,x_{-1},x_0,f(x_0) \big).
\end{equation}
By $\pi_{-k}$ we shall denote
the \emph{$(-k)$-th projection} 
from 
$\underleftarrow{\lim}\{ X,f\}$ to the $(-k)$-th coordinate.  Now we fix 
some notation useful 
for constructing 
parametrised families of examples. We will follow
mainly~\cite{BCH}, and we refer to it
 for the more general setting. 

For $n\geq 1$, 
let $D_n\subset \mathbb{S}^2$ 
denote a closed topological disk with 
$n$ open holes. A subset 
$G\subset D_n$ 
is called a \emph{boundary retract of $D_n$} if there is a continuous map 
$\alpha:\partial D_n\times [0,1]\to D_n$ 
which decomposes $D_n$ into a continuously varying family of arcs 
$\{\alpha(x,\cdot)\}_{x\in \partial D_n}\subset \mathcal C([0,1], D_n)$, 
so that $\alpha(x,\cdot)([0,1])$ 
are pairwise disjoint except 
perhaps at the endpoints $\alpha(x,1)$, where $\alpha(x,1)\in G$. 
We can then associate a retraction  
$r: D_n \to G$ defined by 
$r(\alpha(x,s))= \alpha(x,1)$
for every 
$x\in \partial D_n$ corresponding to the given decomposition. 
We say a continuous map $f: G\to G$ \emph{unwraps in $D_n$} 
if there is a near-homeomorphism $\bar{f}: D_n\to D_n$ such that 
$r\circ \bar{f}=f$. 
The near-homeomorphism $\bar{f}$ is called the unwrapping of $f$. 
Let $I=[0,\epsilon_0]$ for some $\epsilon_0>0$.
A continuous family $\{f_t\}_{t\in I}$ is said to \emph{unwrap} in $D_n$ 
if there exists a continuous family of unwrappings $\{\bar{f_t}\}_{t\in I}$ 
associated to it. We are now ready to state 
the parametrised BBM technique, which 
we will use in our construction later. The following lemma is an adaptation of Theorem 3.1 from \cite{BCH}.

\begin{lemma}\label{BBM}
	For $n\geq 1$, let $D_n$ denote the closed set obtained by removing from a closed disk interiors of $n$ disjoint closed disks. 
	Let 
	$G \subset D_n$ denote 
	a boundary retract of $D_n$ 
	and suppose 
	a family $\{f_t\}_{t\in I}\subset \mathcal{C}(G,G)$ 
	unwraps in $D_n$.  Moreover, suppose that there exists $m>0$ such that $f^{m+1}_t(G)=f^{m}_t(G)$ for all $t\in I$.
	Then, there is a continuous family $\{F_t\}_{t\in I}$ in 
	$\mathcal{H}(D_n, D_n)$, 
	such that:
	\begin{itemize}
		\item[(a)] For each $t\in I$ there is a compact 
		$F_t$-invariant set  $K_t\subset D_n$ so that:
		\begin{itemize}
			\item[(i)] $F_t|_{K_t}: K_t\to K_t$  
			is topologically conjugate to $\sigma_{f_t}: 
			\underleftarrow{\lim} \{G, f_t \} \to \underleftarrow{\lim} \{G, f_t\}$.
			\item[(ii)] If $x\in D_n \setminus \partial D_n$, 
			then the omega limit set 
			$\omega(x, F_t)\subset K_t$. 
		\end{itemize}
		\item[(b)] The attractors $K_t$ vary continuously in Hausdorff metric with $t\in I$. 
	\end{itemize}
\end{lemma}

\begin{proof}[Sketch of proof]\label{smash}
	The proof follows the proof of Theorem 3.1 in the paper~\cite{BCH}. Our version 
	is a bit more general, in that, we allow our unwrapping to be a {\em near-homeomorphism}
	(i.e., the uniform limit of homeomorphisms), 
	instead of just a homeomorphism (this is allowed due to Brown's approximation theorem from \cite{Brown}). 
	The {\em smash} mapping will be also chosen carefully, instead of being fixed as in the original proof. 
In particular, our specific choices of the unwrapping and smash mappings
will imply that the dynamics of 
$F_t$ in $D_n$ 
of examples from Theorem~\ref{parametrised} 
is indeed close to identity in the $C^0$ topology 
as $t$ is sufficiently close to $0$. 
However, the 
same conclusion of the above mentioned theorem holds true 
with our choices of unwrapping and smash mappings with 
proofs unchanged,
so we do not repeat them. 
Instead, we will stress this 
point again during the proofs of the main theorem, 
and give precise definitions of the choices. 
\end{proof}

 Let us remark that whenever we will apply Lemma~\ref{BBM}, the condition $f^{m+1}_t(G)=f^{m}_t(G)$ for all $t\in I$ from the statement of Lemma~\ref{BBM} will be satisfied for $m=1$, so we are not repeating this condition again.

\subsection{Surface Dynamics and Prime Ends Rotation Numbers}
\noindent
We will mainly work with $ \mathbb{S}^2$ due to its compactness, and due to the fact that any planar homeomorphism can be extended to the sphere by compactifying the plane by a point at infinity, and setting it as a fixed point. 
Recall that
$\Gamma$ is called a \textit{translation line} 
for a homeomorphism 
$h:\mathbb{S}^2\to\mathbb{S}^2$,
if there exists a continuous 
injective map $\gamma: \mathbb{R}\to\mathbb{S}^2$, 
onto its image $\Gamma=\gamma(\mathbb{R})$, such that 
$\Gamma$ is 
 $h$-invariant,
 and the restriction $h|\Gamma$ is fixed point free. 
Equivalently, it means that 
the composition 
$\gamma^{-1}\circ h\circ \gamma$ 
is topologically conjugate to the translation 
given by $T(x)=x+1$ for all $x\in\mathbb{R}$. 
Define the  $\omega$-limit of a translation line 
$\Gamma=\gamma( R)$ 
as $\omega(\Gamma)=\bigcap_{t>0}\overline{\gamma([t,\infty))}$. 
If $\Gamma$ is disjoint from $\omega(\Gamma)$, 
define the filled $\omega$-limit set, 
written as 
 $\widetilde{\omega}(\Gamma)$, 
 as the union of $\omega(\Gamma)$
 with all the connected components of 
 $\mathbb{S}^2\setminus\omega(\Gamma)$ which does not 
 contain $\Gamma$.  
 Note that the set $\widetilde{\omega}(\Gamma)$ is a continuum. 
 which 
 does not separate $ \mathbb{S}^2$. 
 Following \cite{Korop} we call a continuum $K$ a \emph{rotational attractor} 
 if it is a topological attractor 
 for $f$,
 and the corresponding 
 external prime ends rotation number is nonzero 
 (modulo $1$). 

We now introduce the terminology from the prime end theory that we will use in the paper. For a comprehensive introduction to the prime end theory the reader is referred to e.g. \cite{Mather}.
It is well known that, for a domain in $\mathbb{S}^2$
which is homeomorphic to an open topological disk $\mathbb{D}$, 
one can define the so-called \emph{prime ends circle}, so that its union with $\mathbb{D}$ 
is homeomorphic to the closed unit disk with a proper topology. 
If $h:\mathbb{S}^2\to \mathbb{S}^2$ preserves orientation and $h(\mathbb{D})=\mathbb{D}$ then $h$ induces an orientation preserving homeomorphism of the prime ends circle, and therefore it gives  
 a natural 
\emph{prime ends rotation number}. 
It is yet to be completely determined 
how exactly the prime ends rotation number is related with the actual dynamics on $\mathbb{D}$.
We refer to the recent paper~\cite{KLCN} for more details and state of the art in this topic.
Here we are mostly interested in knowing how the prime ends rotation number changes 
along a parametrised family of attractors and homeomorphisms, obtained from BBM embeddings. 

Our more specific context is as follows. 
For a dynamical system $f:\Bbb S^2 \to \Bbb S^2$ preserving a 
non-separating invariant
continuum $K$, 
we will consider the so called \emph{exterior prime ends 
rotation number} $\rho_{\text{ex}}(f,K)$. The complement $K^c$ is $f$-invariant, and is 
a topological disk in $ \Bbb S^2$. By definition, 
the exterior prime ends rotation number 
$\rho_{\text{ex}}(f, K)$ is the rotation number of the induced 
 prime ends circle homeomorphism. 
In order to relate 
the value of prime ends rotation number with the actual dynamics, 
one has to understand the sets of accessible points of $K$. 
For this purpose we will need 
a recent result obtained 
by Hern{\'a}ndez-Corbato from \cite{Luis}, which we recall below. 

Let $K$ be a plane non-separating continuum in an interior of a closed disk $D$. If $f$ is a homeomorphism of $D$ with $f(K)=K$, then there exists a $\xi\in K$ so that $f(\xi)=\xi$. Then $D\setminus \{\xi\}$ is homeomorphic to a half-open annulus $\mathbb{A}:=\mathbb{S}^1 \times (−\infty, 0]$.
We define a universal cover $\widetilde{\mathbb{A}}:=\mathbb{R}\times (\infty,0]$ of $\mathbb{A}$ as $\pi:\widetilde{\mathbb{A}}\to \mathbb{A}$ given by $\pi((\theta,r))=(e^{2\pi i \theta},r)$.
Then the lower boundary 
of $\widetilde{\mathbb{A}}\backslash \pi^{-1}(K)$
induces a {\em prime ends line}. If we fix a lift 
$\widetilde f:  \widetilde{\mathbb{A}}\to \widetilde{\mathbb{A}}$ then the induced dynamics on this prime ends line 
has a rotation number, denoted as $\widetilde{\rho}_{\text{ex}}(\widetilde f, K) \in \Bbb R$. 
Clearly,
$\widetilde{\rho}_{\text{ex}}(\widetilde f, K)$ modulo $\Bbb Z$
is the exterior prime ends rotation number 
$\rho_{\text{ex}}(f,K)$.
Therefore, 
we call $\widetilde{\rho}_{\text{ex}}(\widetilde f, K)$ the \emph{lifted exterior prime ends rotation number}. 
A point $p\in K$ is \emph{accessible} if there is an arc 
$\lambda:[0,1]\to \Bbb S^2$, such that $\lambda([0,1))\subset \Bbb S^2\backslash K$ and 
$\lambda(1)=p\in K$. 

\begin{lemma}[Theorem 1.1 and Theorem 1.2 in \cite{Luis}]\label{Luis_lemma}
Let $\widetilde f$ be defined as above. 
Let  
$p\in K$ denote
an accessible point of $K$ from exterior.
Suppose one of the following conditions hold. 
\begin{enumerate}
\item either $p$ is $f$-periodic for some period $n$.
\item or,
 the forward rotation number of $p$ 
 equals the backward rotation number, i.e., for any lifted point $\widetilde p$ corresponding to $p$,
\begin{equation}\label{forward_equals_backward}
\lim_{n\to +\infty} \frac 1n \text{pr}_1 \big(\widetilde f^n(\widetilde p)-\widetilde p\big)= 
\lim_{n\to +\infty} \frac 1n \text{pr}_1\big(\widetilde p-\widetilde f^{-n}(\widetilde p) \big),
\end{equation}
where $\text{pr}_1$ denotes the first coordinate projection. 
\end{enumerate}
Then, the lifted exterior 
prime ends rotation number $\widetilde{\rho}_{\text{ex}}(\widetilde f,K)$ equals 
the point-wise rotation number of the point $p$.
\end{lemma}

We will also need the following result by Barge~\cite{Barge}. 
\begin{lemma}[Proposition 2.2 in \cite{Barge}]\label{Barge}
Suppose that $\{\Phi_t\}_{t\in I}$ is a continuous family of orientation-preserving  homeomorphisms on $\Bbb S^2$. 
For every $t\in I$ let $K_t$ be a 
non-degenerate sphere non-separating continuum, invariant under $\Phi_t$,
and assume that $\{K_t\}_{t\in I}$ vary continuously with $t$ in Hausdorff metric. 
Then the exterior prime ends rotation numbers 
$\rho_{\text{ex}}(\Phi_t,\Lambda_{t})$ vary continuously with $t$.   	
\end{lemma}

\section{Auxiliary Lemma on Circle Endomorphisms with Two Turns}\label{endo_rotation_lemmas}

In this section, we recall a useful lemma concerning rotation sets of circle endomorphisms with two turns. This should be already known from~\cite{Ito_rotation} and~\cite{Palis_Endo} (see also \cite{Boyland_1986_bifurcations} for this special class of endomorphisms). For reader's convenience we include the proofs. Denote by $\text{End}_1(\mathbb{S}^1)$ the space of circle endomorphisms of degree $1$, and by 
$\widetilde{\text{End}}_1 (\mathbb{S}^1)$ the set of lifts of elements in 
$\text{End}_1(\mathbb{S}^1)$ to $ \mathbb{R}^1$. 
For any $\widetilde f\in \widetilde{\text{End}}_1(\mathbb{S}^1)$, 
we define a rotation set $\rho(\widetilde f)$ as follows. 
First let
\begin{equation}\label{original_limsup_definition}
\rho(\widetilde f, x):=\limsup_{n\to \infty} \frac 1n (\widetilde f^n(\widetilde x)-\widetilde x).
\end{equation}
Note that the expression above 
does not depend on the choice of the lifted point $\widetilde x$  corresponding to $x$.
We set  
\begin{equation}\label{eq:rho}
\rho(\widetilde f): = \{ \rho(\widetilde f, x) \big| 
x\in \mathbb{S}^1\},
\end{equation} which 
is a closed interval by the main result of \cite{Ito_rotation}. Now let us consider a family of endomorphisms of a special form, which was also considered in Section 2 of \cite{Boyland_1986_bifurcations} for other purposes. 
Suppose some interval $[\widetilde z_0,\widetilde z_0+1]$ 
is subdivided into two 
subintervals, namely, 
$\widetilde{I}_1= [\widetilde z_0, \widetilde y_0]$, 
$\widetilde{I}_2=[\widetilde y_0,\widetilde z_0+1]$. 
Assume $\widetilde f\in \widetilde{\text{End}}_1(\mathbb{S}^1)$
is such that $\widetilde f\big |_{\widetilde{I}_1}$ is decreasing, 
and $\widetilde f\big |_{\widetilde{I}_2}$ is increasing. 
Whenever there exist $\widetilde z_0,\widetilde y_0$ as above, 
we call such a $\widetilde f$ the lift of a \emph{circle endomorphism with two turns.}
For such an endomorphisms it was proved in \cite{Boyland_1986_bifurcations} that every rotation 
number $\rho \in \rho (\widetilde f)$ can always be realized by some point. More precisely, 
we can replace the original definition (\ref{eq:rho})
with the more naturally defined pointwise rotation number. 
\begin{equation}\label{point_wise_definition}
\rho_{\text{pp}}(\widetilde f) = \{ \lim_{n\to \infty} \frac 1n (\widetilde f^n(\widetilde x)-\widetilde x) \big | \text{ when the limit exist.}\}
\end{equation}
By Proposition 2.3 of \cite{Boyland_1986_bifurcations}, for an endomorphism with two turns we have 
\begin{equation}\rho_{\text{pp}}(\widetilde f)= \rho(\widetilde f).
\end{equation}
We shall need this observation in the arguments below. Let us define the point 
\begin{equation}
\widetilde w_0:= \sup \{\widetilde x \big | \widetilde x \in [\widetilde z_0,\widetilde z_0+1], 
\widetilde f(\widetilde x)=\widetilde f(\widetilde z_0)\}.
\end{equation}
Then, we call the interval 
$\widetilde J:=[\widetilde w_0,\widetilde z_0+1]$ 
an \emph{efficient climbing interval}. Note that it is uniquely defined 
up to an integer translation.
We also define the \emph{lower climbing interval} as the interval 
$[\widetilde y_0, \widetilde w_0']$, where
\begin{equation}
\widetilde w_0'=\inf\{\widetilde x \big| \widetilde x \in [\widetilde y_0,\widetilde z_0+1], \widetilde f(\widetilde x) = \widetilde f(\widetilde y_0) +1\}.
\end{equation}

\begin{lemma}\label{one_turn}
Let $\widetilde f \in\widetilde{\text{End}}_1(\mathbb{S}^1)$ 
be the lift of a circle endomorphism with two turns, and denote by $\widetilde J$ an 
efficient climbing interval of $\widetilde f$.
Then there exists a point $\widetilde y_0\in \widetilde J$, with the following two properties. 
\begin{itemize}
\item It is possible to choose backward iterates of $\widetilde y_0$, namely,  
 $\widetilde y_{-k}\in \widetilde{f}^{-k}(\widetilde {y}_0)$, such that each $\widetilde y_k$
 belongs to some integer translate of $\widetilde J$.

\item The forward rotation number of $\widetilde y_0$
coincides with the backward 
rotation number of $\widetilde y_0$, which is the supremum 
of the rotation segment $\rho(\widetilde f)$.
\end{itemize}
\end{lemma}
\begin{proof}
For endomorphisms with two turns, since the definitions 
(\ref{eq:rho}) and (\ref{point_wise_definition}) are equivalent, 
there always exists some point $y_0$ with a lift $\widetilde y_0$, 
 such that,
 \begin{equation}\label{forward_rot}
 \sup \rho(\widetilde f)=\rho(\widetilde f, y_0)
 =\lim_{n\to +\infty}\frac 1n (\widetilde f^n(\widetilde y_0)-\widetilde y_0).
 \end{equation} 
 Under the assumptions, 
the image of $\widetilde J$ is the interval 
$[\widetilde f(\widetilde z_0),\widetilde f(\widetilde z_0+1)]$, 
which has length $1$, because $f$ has degree $1$.
It follows that
$\underset{m\in  \mathbb{Z}}{\bigcup}\widetilde f(\widetilde J+m)= \mathbb{R}$. 
Thus for any $\widetilde z\in  \mathbb{R}$, there exists an integer $m$, 
such that 
$\widetilde f^{-1}(\widetilde z)\cap (\widetilde J+m)\neq \emptyset$. 
In particular, 
for $\widetilde y_0 \in \widetilde J$, 
we can choose a sequence of its backward iterates by $\widetilde f$, 
each of which lies in the translates of $\widetilde J$. Therefore, up to taking one backward iterate, 
we can choose backward iterates
$\{\widetilde y_{-k}\}_{k\geq 0}$
such that for all $k\geq 0$, 
$\widetilde y_{-k}$  
belongs to the translates of $\widetilde J$, 
and $\widetilde f(\widetilde y_{-(k+1)})
=\widetilde y_{-k}$.
Now we modify our dynamics. 
Define the function 
$\widetilde g: \mathbb{R}\to  \mathbb{R}$, which 
coincides with $\widetilde f$ when restricted to $\widetilde J=[\widetilde w_0,\widetilde z_0+1]$, and 
takes constant value 
$\widetilde f(\widetilde z_0)$ when restricted to the interval $[\widetilde z_0, \widetilde w_0]$ (the map $\widetilde{g}$ is the so-called ''water pouring map``, see \cite{ALM}, page 143 for more details).
Note that, $\widetilde g$ is a monotone function, so its rotation number is uniquely defined,
independent of any starting point. 
Since the backward iterates of $\widetilde y_0$ all lie in translates of $\widetilde J$, we know in particular that the backward rotation number of $\widetilde y_0$ exists
and it is equal to $\rho (\widetilde g)$. Clearly, it follows that 
the forward rotation number of $\widetilde y_0$ equals the backward rotation number of $\widetilde y_0$.
The proof is complete now.
\end{proof}

\begin{remark}\label{climbing_properly_parellel}
Lemma~\ref{one_turn} will be used for showing accessibility of certain 
points in BBM embeddings of some circle-like attractors, and this in turn will serve as a way of 
determining their exterior prime ends rotation numbers. 
The definition of the efficient climbing interval is designed for this purpose. 
Note that, with the definition of lower climbing intervals, 
one can similarly study the interior prime 
ends rotation numbers, and obtain a similar statement of Lemma~\ref{one_turn}. 
We omit the repetition of the proofs. 
However, we will use both notions in 
Subsection~\ref{Walker}.
\end{remark}

\section{A parametric family near the identity}\label{lakes_wada}
In this section we prove Theorem~\ref{parametrised}.
Let us first introduce the following notion, which will be used in the present and the subsequent sections. 
Call a connected topological graph $G_k$ 
\emph{a chain of $k$ circles} if $G_k$ is a union of 
$k>1$ circles
$S_0\cup\ldots\cup S_{k-1}$ where for any $i,j\in \{0,1,\ldots, k-1\}$,
$S_i\cap S_j$ is a point if and only if $|i-j|=1$ and if $|i-j|>1$, then $S_i\cap S_j=\emptyset$.

\begin{proof}[Proof of Theorem \ref{parametrised}]
	We start with the case $n=3$, i.e. 
	we construct a $3$-separating Lakes of Wada continuum. \\

\noindent{\textbf{Step 1. Boundary Retracts and the choice of $\epsilon_0$}}.\\

Let $G_2$ be a chain of $2$ circles $S_0$ and $S_1$ 
with the same radius 
whose intersection contains a single point $x_0$. 
Then $G_2$ is the spine of a pair of pants $D_2$.
Topologically, 
$D_2$
is obtained from a closed topological disk, with boundary $C_2$, 
by removing two disjoint open disks from it, with boundaries $C_0$ and $C_1$ respectively, see Figure~\ref{fig_pants}.
Moreover, $G_2$ is the image of a boundary retract of $D_2$.
 More precisely, we can define the function 
 $\alpha: \partial D_2 \times [0,1] \to D_2$
 such that, for any $x\in \partial D_2$, 
 $\alpha(x, \cdot)$ restricted to $[0,1]$
 is one-to-one and $\alpha(x,0)=x$ 
 and $\alpha(x,1)\in G_2$. For our convenience, 
 we will always choose $\alpha$, such that 
 the image of the arc $\alpha(x, \cdot) ([0,1])$ is straight line segments connecting 
 $\partial D_2$ to $G_2$ (see Figure~\ref{fig_pants}). We refer these as the \emph{radial directions}.
 \begin{figure}[!ht]
	\begin{center}
		\includegraphics[width=8cm]{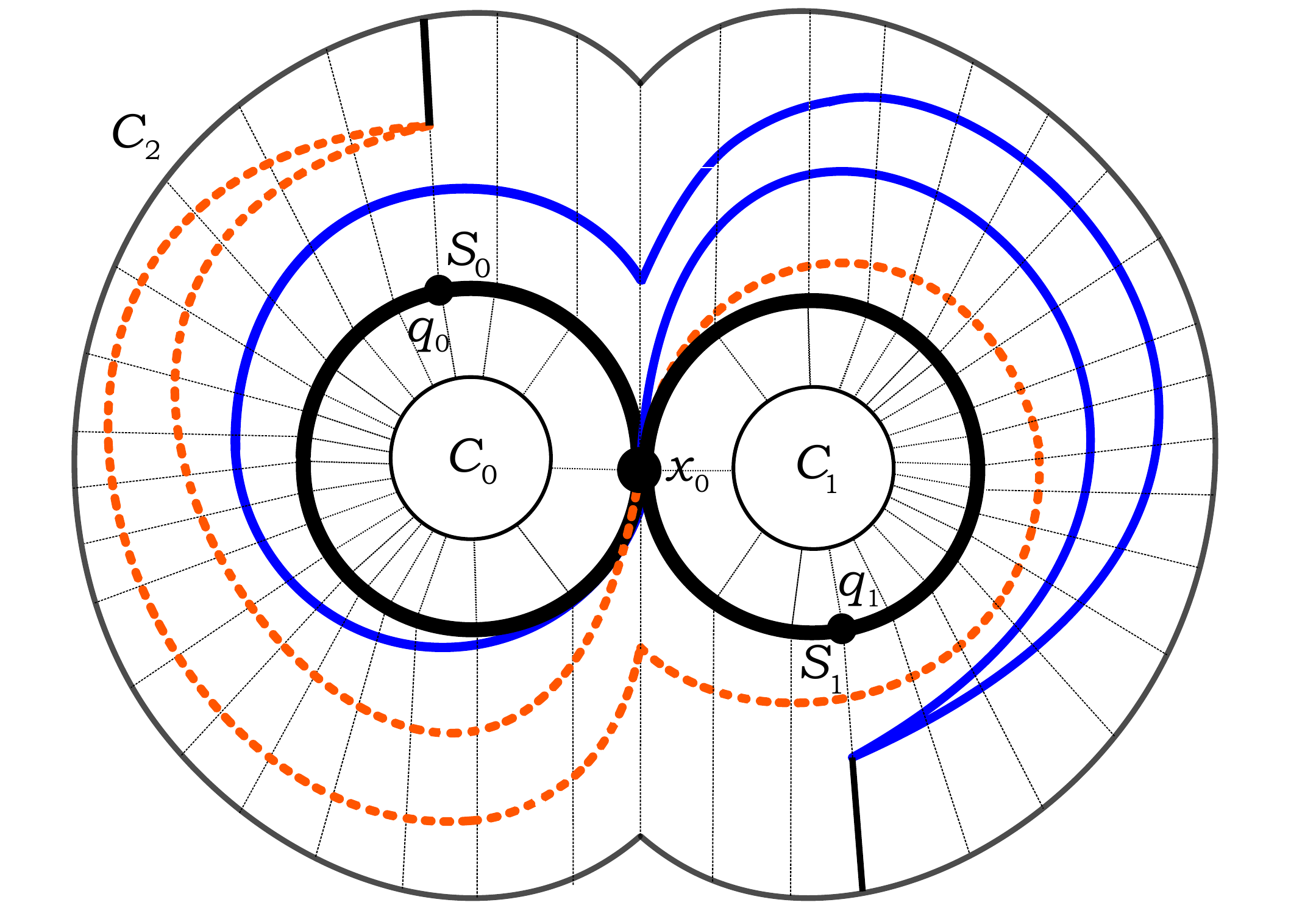}
		\put(-152,140) {$\overline{\phi}(Q_1)$}
		\put(-108,12) {$\overline{\phi}(Q_2)$}
		\caption{The radial decomposition of $D_2$ and planar representation of the map $\overline{\phi}_{\epsilon_0}$ and the arcs $\overline{\phi}_{\epsilon_0}(Q_0)$ and $\overline{\phi}_{\epsilon_0}(Q_1)$.
		The picture visualizes the idea behind the creation of arcs in the inverse limit that access certain points of $K_{\epsilon}$ from the complement of the attractor.} ~\label{fig_pants}
	\end{center}
\end{figure}

For each $i=0,1$, 
we parametrise each circle clockwise 
with 
$\beta_i:[0,1]\to S_i$ in the uniformly scaling way,
so that,
$\beta_i(t_i)=x_0$ if and only if $t_i\in\{0,1\}$. 
Fix $\epsilon \in(0,\frac 12)$. Now we define a continuous map 
$\phi_{\epsilon}: G_2\to G_2$ as follows. 
For $i=0, 1$:
\begin{enumerate}
\item Consider the arc $\beta_i([0, 1-2\epsilon])$. 
The restriction of
$\phi_\epsilon$ to this arc 
is a uniform scaling map, whose scaling factor is $\frac{1}{1-2\epsilon}$. 
Note that the image is 
the whole circle $S_i$. 
\item Consider the arc  
$\beta_i([1-2\epsilon, 1-\epsilon])$.  
The restriction of $\phi_\epsilon$ maps this arc 
to the arc  
$\beta_{1-i}([0, \frac{\epsilon}{1-2\epsilon}])$, in a uniform scaling way.
\item Consider the arc 
$\beta_i([1-\epsilon, 1])$. We require that 
$\phi_\epsilon$ maps this arc to the arc 
$\beta_{1-i}([\frac{\epsilon}{1-2\epsilon}, 0])$, in a uniform scaling way. 
Note that this restriction reverses the orientation.
\end{enumerate}

In particular, 
 for $\epsilon_0= 1-\frac{\sqrt 2}{2}$,
 the two points 
 $q_i=\beta_i(\frac {\sqrt 2}{2})$, with $i=0,1$, 
 form 
 a periodic orbit for $\phi_{\epsilon_0}$ whose period is $2$.
   We denote by $I:=[0,\epsilon_0]$. \\
  
\noindent {\textbf{Step 2. The Choice of the Smash Mapping and the 
Unwrapping}}. \\

Fix some $\epsilon \in (0,\epsilon_0]$. 
Let $\alpha : \partial D_2 \times [0,1] \to D_2$ 
be as given in the previous step.
Define 
$\gamma_\epsilon: [0,1]\to [0,1]$, such that $\gamma_\epsilon(s)=\frac{s}{1-\epsilon}$ for $s\in[0,1-\epsilon]$,
and $\gamma_\epsilon(s)=1$ for $s\in [1-\epsilon,1]$.
Then we 
define the smash mapping $\Upsilon_\epsilon: D_2\to D_2$ as follows. 
 For every 
$x\in \partial (D_2)$ and $s\in[0,1]$ let
\begin{equation}
\Upsilon_\epsilon(\alpha(x,s))=\alpha(x,\gamma_\epsilon(s)).
\end{equation}
Note that the ``smashing region'' for $\Upsilon_\epsilon$ 
is the set 
\begin{equation}\label{omega_epsilon_set}
\Omega_\epsilon: = \{ \alpha(x,s) \big | x\in \partial (D_2), s \in[1-\epsilon,1] \}.
\end{equation}
There is a natural homeomorphism 
$\varphi_\epsilon: D_2 \to \Omega_\epsilon$, 
defined by 
\begin{equation}
\varphi_\epsilon\big(\alpha(x,s) \big):=\alpha(x,\epsilon s+1-\epsilon).
\end{equation}
Note $\gamma_\epsilon$ is a near-homeomorphism of $[0,1]$, i.e., a uniform limit of homeomorphisms. 
It follows that 
$\Upsilon_\epsilon$ 
is a near-homeomorphism of $D_2$. 
The function $\Upsilon_\epsilon$ represents the {\em smash} 
 in the BBM construction and the smashing region $\Omega_\epsilon$ converges to $G_2$
 as $\epsilon$ tends to $0$. 
 In particular, 
 $\Omega_\epsilon$
 is sufficiently thin if $\epsilon$ is sufficiently small. 
 We stress again that 
 there is a difference between this setup and the original proof of 
 the Theorem 3.1 of~\cite{BCH}, 
 where the 
 definition of the smash mapping 
 $\Upsilon$ was fixed
 for the whole parametrised family. Nevertheless, 
 the parametrised family 
  $\{\phi_{\epsilon}\}_{\epsilon\in I} \subset \mathcal{C}(G_2,G_2)$ 
  unwraps in $D_2$. 
  We proceed 
 by describing the specific choice 
 of the definition of
 the unwrapping $\overline{\phi}_\epsilon$.
 
 \begin{definition}\label{overlinephi_definition}
 	Define $\overline{\phi}_\epsilon \in \mathcal C(D_2,D_2)$ as a near-homeomorphism, satisfying 
 	the following conditions.
 	\begin{enumerate}[(a)]
 		\item  $\overline {\phi}_\epsilon(x)=x$ for $x\in C_0\cup C_1$.
 		\item $\alpha(C_0\cup C_1,[0,1))$ is a homeomorphism onto its image. 
 		In particular, $\overline{\phi}_{\epsilon}|_{G_2}$ is a homeomorphism onto its image.
 		\item $\overline{\phi}_{\epsilon}({G_2})\subset \alpha(C_2,[0,1))$.
 		\item The restriction of $\overline{\phi}_{\epsilon}$ to $G_2$ satisfies
 		\begin{equation}
 		\phi_\epsilon = \Upsilon_\epsilon \circ \overline {\phi}_\epsilon \big|_{G_2}.
 		\end{equation}
 		\item 
 		Denote $J_{i,\epsilon}=\beta_i([\epsilon, 1-\epsilon])$, 
 		for $i=0,1$, 
 		and then denote $J_\epsilon=J_{0,\epsilon} \cup J_{1,\epsilon}$.
 		For any $y\in J_\epsilon$, let $x\in C_2$ be such that $\alpha(x,1)=y$. Denote by $z\in C_2$ a point so that $\alpha(z,1)=\phi_\epsilon(y)$.
 		Then $\overline{\phi}_\epsilon$
 		restricted to the 
 		radial arc $\alpha(x,[0,1])$ is a monotone map, 
 		whose image is contained in
 		a radial arc connecting $\phi_\epsilon(y)$ to $z$. 
 	\end{enumerate}
 \end{definition}
 See Figure~\ref{unwrapping_definition} for certain arcs and their images under $\overline{\phi}_\epsilon$.
 
 \begin{figure}[!ht]
 	\begin{center}
 		\includegraphics[width=8cm]{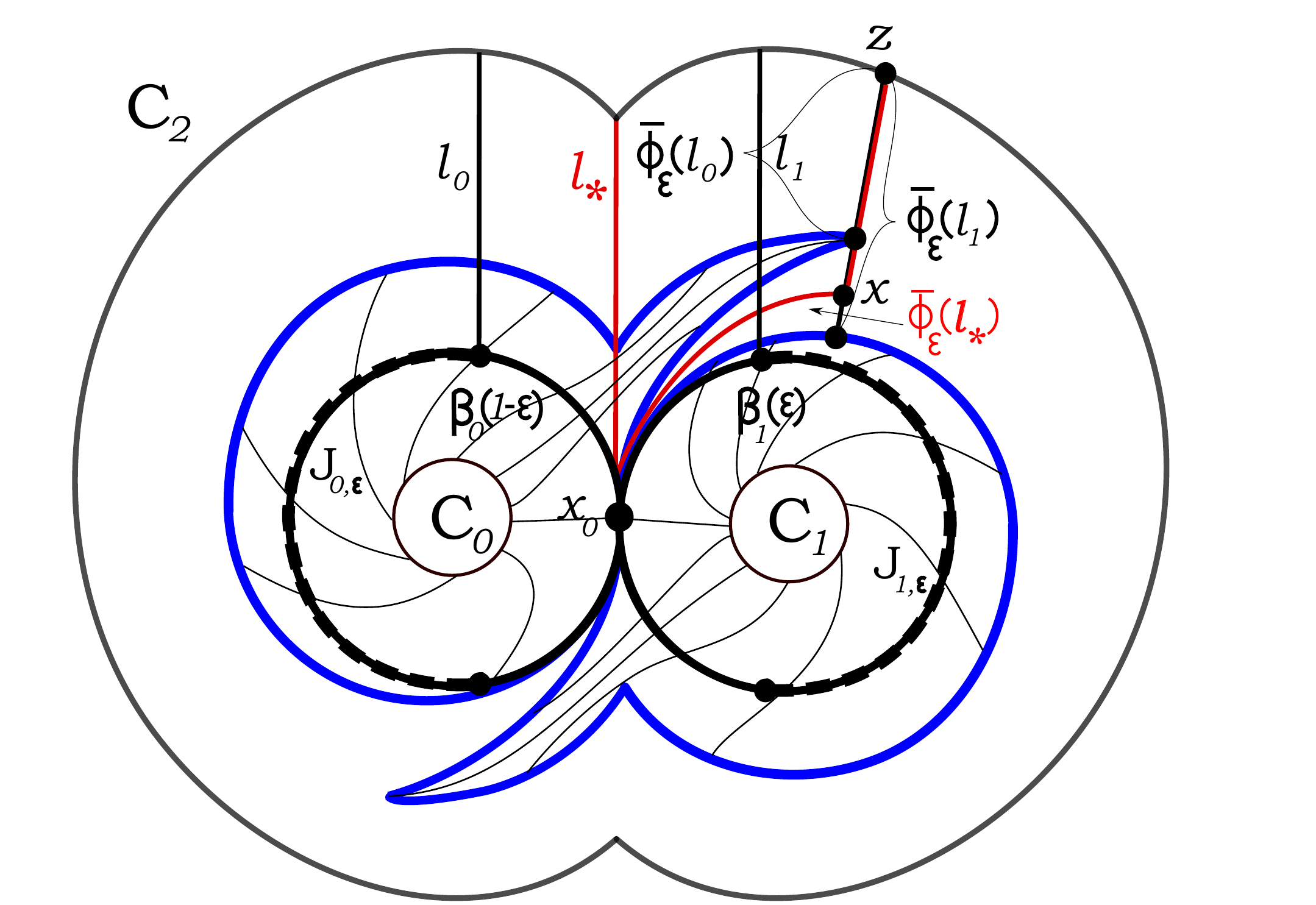}
 		
 		\caption{Picture explains the unwrapping $\overline{\phi}_\epsilon$. 
 			The graph of $\overline{\phi}_\epsilon (G_2)$ is drawn in blue.
 			The arc $\ell_0$ (respectively $\ell_1$) is the radial arc connecting $\beta_0(1-\epsilon)$ (respectively $\beta_1(\epsilon)$) and its corresponding radial point in $C_2$. 
 			The arc $\ell_\ast$ is the upper of the two radial arcs connecting $x_0$ to its corresponding radial point in $C_2$. 
 			By definition, we require that 
 			$\overline{\phi}_\epsilon(\ell_0)$ is the arc connecting the point 
 			$\overline{\phi}_\epsilon(\beta_0(1-\epsilon))$ radially to $z\in C_2$.
 			The arc $\overline{\phi}_\epsilon(\ell_1)$ is the radial arc 
 			from $\overline{\phi}_\epsilon(\beta_1(\epsilon))$ 
 			to $z$, and contains the arc $\overline{\phi}_\epsilon(\ell_0)$.
 			The arc $\overline {\phi}_\epsilon(\ell_\ast)$ is the concatenation of two arcs. 
 			The first arc connects $x_0$ to $x\in \overline{\phi}_\epsilon(\ell_1)\setminus \overline{\phi}_\epsilon(\ell_0)$ and is depicted in red here. The second arc connects $x$ to $z$ and is a subarc of $\overline{\phi}_\epsilon(\ell_1)$. Intervals $J_{0,\epsilon}$ and $J_{1,\epsilon}$ are drawn on the picture by dashed lines.}
 		\label{unwrapping_definition}
 	\end{center}
 \end{figure}

 \begin{remark}
 	In item (e) of the Definition~\ref{overlinephi_definition}, the choice of the interval $J_\epsilon$ is 
 	related to the efficient climbing interval that we have defined in 
 	Section~\ref{endo_rotation_lemmas} for circle endomorphisms with two turns. 
 	Rigorous proof that $\overline {\phi}_\epsilon$
 	is indeed a near-homeomorphism, is left to the reader.
 	Let us note that $\overline{\phi}_{\epsilon}$ can be defined to be a homeomorphism as well (see Figure~\ref{fig:unwrappingHomeo}), but for our study of prime ends it is more convenient to use a near-homeomorphism that we define above.
 \end{remark}

\begin{figure}[!ht]
	\begin{center}
		\includegraphics[width=8cm]{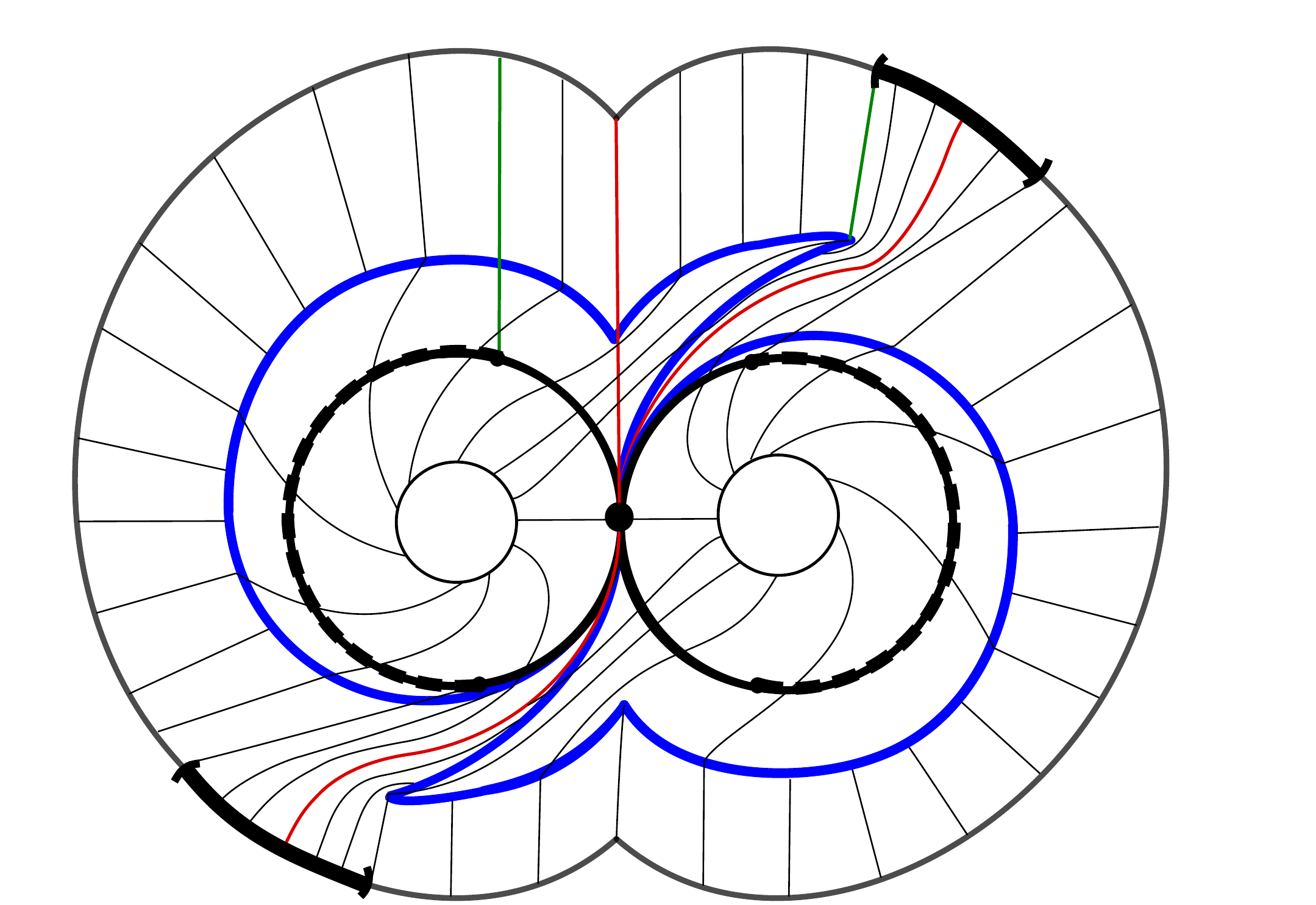}
\end{center} 
\caption{This picture explains that an unwrapping of $\phi_{\epsilon}$ can be a homeomorphism of $D_2$ as well. However, in this case it becomes much harder to determine accessible points of $K_{\epsilon}$. In particular, the proof of Theorem~\ref{parametrised} that we present here does not work, since radial arcs do not get mapped to radial arcs.}
\label{fig:unwrappingHomeo}
\end{figure}

 As we already remarked in the proof of Lemma~\ref{BBM}, 
although the unwrapping 
$\overline{\phi}_\epsilon$ is by the definition 
only a near-homeomorphism, the conclusions of 
Lemma~\ref{BBM} still hold and were used in this form in e.g. \cite{Boy} for studying unimodal inverse limit spaces as attractors of sphere homeomorphims.
 Thus, similarly as in the beginning of page 1081 of the paper \cite{BCH} (but
 with our choice of $\Upsilon_\epsilon$), 
 we can define 
 the mapping 
 \begin{equation}\label{fepsilon_bar}
 \overline{f}_\epsilon: = \varphi_\epsilon \circ \overline {\phi}_\epsilon\circ \varphi_\epsilon^{-1}.
 \end{equation} 
 Then we extend it to the whole $D_2$ radially. 
 Finally, 
consider the composition 
\begin{equation}\label{definition_of_psi}
\psi_\epsilon:= \Upsilon_\epsilon \circ \overline{f}_\epsilon,
\end{equation} 
and denote 
the induced shift homeomorphism by 
 \begin{equation}
 \Phi_\epsilon: = \sigma_{\psi_\epsilon}.
 \end{equation}
By Brown's approximation theorem from \cite{Brown},
the inverse limit space 
 $\underleftarrow{\lim} \{D_2,  \psi_\epsilon\}$
is homeomorphic to $D_2$ itself, via a homeomorphism $h_\epsilon$. 
Note that it follows from item (b) of 
Definition~\ref{overlinephi_definition} that 
$\phi_\epsilon\big|_G= \psi_\epsilon \big |_G$. 

From now on, we will identify via $h_\epsilon$ the spaces 
$\underleftarrow{\lim} \{D_2,  \psi_\epsilon\}$ with $D_2$. 
Denote $K_\epsilon:=h_{\epsilon}(\underleftarrow{\lim}\{G, \phi_\epsilon\})$. 
Then, in order to shorten the cumbersome notations,
we will neglect the homeomorphism $h_{\epsilon}$ and work 
with $K_{\epsilon}$ by identifying it with $\underleftarrow{\lim}\{G, \phi_\epsilon\}$.
Finally, we observe that $\psi_\epsilon$
extends to $ \mathbb{S}^2$, and we are not concerned about 
what happens outside 
 $D_2$.
So in what follows, it suffices for us to 
only consider $\psi_\epsilon$ and $\Phi_\epsilon$ restricted to 
$D_2.$ \\

 \noindent{\textbf{Step 3. Main arguments of the proof.}}\\
 
 We start to check the assertions of Theorem~\ref{parametrised}.
 As explained in the previous paragraph, 
 by Lemma~\ref{BBM}, 
 the $\Phi_\epsilon$-invariant continuum 
  $K_\epsilon$ is an attractor, 
  and $\Phi_{\epsilon}|_{K_{\epsilon}}$ 
  is topologically conjugate to the shift homeomorphism 
  $\sigma_{\phi_{\epsilon}}$ for every $\epsilon \in I\setminus \{0\}$.
  Furthermore, by \text{(b)} from Lemma~\ref{BBM},
  $K_{\epsilon}$ vary continuously 
  in the Hausdorff distance with parameter $\epsilon\in I$. 
  This shows item (i) of Theorem~\ref{parametrised}.
    
 Next, in order to understand the prime ends rotation number, we need to
 first  study 
 accessible points of these attractors. First we show how this is done when $\epsilon=\epsilon_0$.
\begin{lemma}\label{claim1}
Recall the choice of $q_0$ and $q_1$ at the end of Step 1 of the proof. 
Let $\underbar{q}:=(\ldots,q_0,q_1, q_0),
      \underbar{q}':=(\ldots,q_1,q_0,q_1) 
      \in K_{\epsilon_0}$. 
      Then, the points 
      $\underbar{q}$ and $\underbar{q}'$ 
      are accessible points of $K_{\epsilon_0}$. 
\end{lemma}
\begin{proof} 
Choose the smash mapping $\Upsilon_{\epsilon_0}$ and 
the unwrapping 
 $\overline{\phi}_{\epsilon_0}$, and obtain 
 the near-homeomorphism $\psi_{\epsilon_0}$
  as in Step 2. 
 Denote 
 $Q_0,Q_1\subset D_2$ for the 
 radial arcs from $q_0,q_1$ to the boundary component 
 $C_2$, respectively. Then the following holds:
\begin{align}
 \psi_{\epsilon_0}(Q_0)
     =Q_1.\\
 \psi_{\epsilon_0}(Q_1)=Q_0.
\end{align}

Let us focus on $i=1$, 
and consider
 the inverse limit set 
\begin{equation}
\underbar{Q}:= \big(\ldots, Q_1, Q_0, Q_1\big).
\end{equation} 
Note that, 
$\psi_{\epsilon_0} \big|_{Q_i}$ 
is a near-homeomorphism for every $i=0,1$. 
This shows that $\underbar{Q}$ 
is an arc in the inverse limit 
space $\underleftarrow{\lim}\{D_2, \psi_{\epsilon_0} \}$, 
which is identified with $D_2$.
Moreover, for any $\underbar x\in \underbar{Q}\backslash \{\underbar q'\}$, by definition, 
for some $k>0$, $\pi_{-k}(\underbar x)\notin G_2$. 
This implies that, 
for all $m\geq k$, $\pi_{-m}(\underbar x) \notin G_2$. Therefore, 
$\underbar{Q}
 \cap K_{\epsilon_0}=
 \underbar{q}'$, and so 
$\underbar{q}'$ is accessible by $\underbar{Q}$.
The argument for the point $\underbar{q}$  follows analogously.
\end{proof}

 For any fixed 
 $\epsilon\in (0, \epsilon_0]$, let $\bar{K}_{\epsilon}$ denote the union of 
the attractor $K_{\epsilon}$ with its two complementary domains with boundaries $C_0$ and $C_1$ respectively (i.e. we are considering filled disk $D_2$). 
 Then, $\bar{K}_{\epsilon}$ 
 is a plane non-separating continuum whose boundary 
 is just $K_{\epsilon}$. So we can talk about 
 the exterior prime ends rotation number of $\bar K_\epsilon$, 
 and denote it as $\rho_{\text{ex}}(\Phi_\epsilon,K_\epsilon)$.
 For the parameter $\epsilon_0$,  
 since we found two accessible periodic points $\underbar{q}_0$ and $\underbar{q}_1$ 
of $K_{\epsilon_0}$, Lemma~\ref{Luis_lemma} implies that   the exterior prime ends rotation number satisfies
  $\rho_{\text{ex}}(\Phi_{\epsilon_0}, K_{\epsilon_0})=\frac{1}{2} \text{ mod } \Bbb Z$.

 In fact, Lemma~\ref{claim1} is a simpler version of what we will do next. 
The new difficulty is that, we do not have accessible periodic orbit in general.
 
 Recall the choices of the smash mapping $\Upsilon_{\epsilon}$
 and the unwrapping $\overline{\phi}_\epsilon$ in Step 2. 
 Let us mark an arbitrarily point
  $q_\ast\in G_2$, and let the 
  radial arc $Q_\ast$ connect $q_\ast$ 
  to a point in $C_2$ where $Q_\ast\cap G_2=\{q_\ast\}$.
 We can then 
 choose a sequence of
 backward iterates, namely, $q_{-k}\in \phi_{\epsilon}^{-k}(q_\ast)$, 
 such that for all $k\geq 1$, 
 $q_{-k}\in J_\epsilon$ 
 (the set $J_\epsilon$ was defined 
 in item (d) of Definition~\ref{overlinephi_definition}). 
  Then we claim that the element $\underbar{q}_\ast := (\ldots,q_{-2},q_{-1},q_\ast)$ 
  is an accessible point of $K_\epsilon$.
 
To show the claim, we note that for any $k\geq 1$, 
$q_{-k}$ is the endpoint of a radial arc $Q_{-k}$, 
and $Q_{-k}\cap G_2=\{q_{-k}\}$. 
Observe also that,
for all $k\geq 1$,
$\psi_\epsilon$ restricted to $Q_{-k}$ is a near-homeomorphism onto $Q_{-(k-1)}$. 
Consider the inverse limit space, obtained by these arcs,
$\underbar{Q}_\ast:= (\ldots,Q_{-2},Q_{-1},Q_\ast)$. 
Then, $\underbar{Q}_\ast$ is an arc in the inverse limit space 
$\underleftarrow{\lim}\{D_2, \psi_{\epsilon} \}$, which in turn is homeomorphic via $h_\epsilon$
to $D_2$. 
Similar argument as in the proof of Lemma~\ref{claim1} shows that, 
$\underbar{Q}_\ast \cap K_\epsilon= \{ \underbar{q}_\ast \}$, and 
clearly $\underbar{Q}_\ast \backslash \{\underbar{q}_\ast\}$ is contained in 
the exterior complementary domain of $K_\epsilon$.
Thus,
$\underbar{q}_\ast$ is accessible by $\underbar{Q}_{\ast}$ from the exterior complementary domain of $\bar{K}_{\epsilon}$.

Now we are ready to show item (ii) and (iii) of Theorem~\ref{parametrised}.
Observe first the case $\epsilon=0$. 
Define $\phi_0=\text{id}$, and then $K_0=G_2$ and all the construction becomes trivial, in 
the sense that the smash mapping and the unwrapping are all identity (note that we do not have an attractor at this instance yet). 
In particular, $\rho_{\text{ex}}(\Phi_{0}, K_{0})=0 \text{ mod } \Bbb Z$. 
Then by Proposition~\ref{Barge}, we note 
the exterior 
prime ends rotation numbers
$\rho_{\text{ex}}(\Phi_\epsilon, K_\epsilon)$ vary continuously with $\epsilon\in I$.

Both the smash mapping and the unwrapping 
can be defined in a symmetric way, 
with respect to $S_0$ and $S_1$, due to symmetricity in the definition of the map $\phi_\epsilon$.
Let us consider the continuum 
$\text{Fill}(S_0\bigcup S_1)$, 
obtained by 
the union of $G_2$ with
the interior disks bounded by two circles $S_0$ and $S_1$. 
The boundary $G_2$ 
of the continuum $\text{Fill}(S_0\bigcup S_1)$
can be for our purposes regarded as a single circle
\footnote{Rigorously, 
the circle $\Bbb T^1$ is exactly the exterior prime ends circle of the continuum 
$\text{Fill}(S_0\cup S_1)$.
Therefore, the point $x_0\in S_0
\cap S_1$ corresponds to two 
antipodal points in the circle $\Bbb T^1$. }, denoted by $\Bbb T^1$.
Naturally, the map $\phi_\epsilon$ 
induces a map $\phi_\epsilon^{\ast}$ on $\Bbb T^1$. 
Observe that, 
$\phi_\epsilon^{\ast}$ is a 
two cover of a circle endomorphism 
with two turns 
(see the notation in Section~\ref{endo_rotation_lemmas}). 


Fix any 
$\epsilon\in (0,\epsilon_0]$, and
choose a proper lift $\widetilde{\phi}_\epsilon^{\ast}$.
We see the interval $J_\epsilon$ corresponds to 
exactly two copies of the efficient climbing intervals for the lifted map. 
The restriction of $\widetilde{\phi}_\epsilon^{\ast}$
to $[0,\frac 12]$ consists of one increasing interval and 
one decreasing interval. 
We now apply 
Lemma~\ref{one_turn} to the lifted 
map 
$\widetilde{\phi}_\epsilon^{\ast}$, 
obtaining 
some point 
$p\in J_\epsilon$, such that 
for any lifted point $\widetilde p$ corresponding to $p$, we can choose its backward iterates $\widetilde p_{-k}\in (\widetilde{\phi}_\epsilon^{\ast})^{-k}(\widetilde p)$, with the following properties. 
\begin{enumerate}
\item Each $\widetilde p_{-k}$ belongs to $\pi^{-1}(J_\epsilon)$ (recall the definition of $\pi$ in the paragraph preceding Lemma~\ref{Luis_lemma}).
\item The backward rotation number equals the forward rotation number, which realises
the upper endpoint of the rotation interval $\rho\big( \widetilde{\phi}_\epsilon^{\ast} \big)$:
\begin{equation}
\lim_{n\to +\infty} \frac{1}{n} \big( (\widetilde {\phi}_\epsilon^{\ast})^n(\widetilde p) -\widetilde p \big) 
= \lim_{n\to +\infty}\frac{1}{n} \big( \widetilde p -  \widetilde p_{-k}\big)=\sup \rho\big( \widetilde{\phi}_\epsilon^{\ast} \big).
\end{equation}
\end{enumerate}

 Then, 
 as we have already shown, 
 the point $(\ldots,p_{-2},p_{-1},p) \in K_\epsilon$ 
is a point accessible from the complement of $\bar{K}_{\epsilon}$. 
Thus, by Lemma~\ref{Luis_lemma}, we can consider the lifted prime ends rotation number 
$\widetilde{\rho}_{\text{ex}}(\widetilde \Phi_\epsilon,K_\epsilon)$, and it follows that 
\begin{equation}
    \widetilde{\rho}_{\text{ex}}(\widetilde \Phi_\epsilon,K_\epsilon)
=  \sup \rho\big( \widetilde {\phi}_\epsilon^{\ast} \big).
\end{equation} 
Now, the monotonicity of the function $\frac{1-\epsilon}{1-2\epsilon}$ implies the monotonicity of 
$\widetilde {\phi}_\epsilon^{\ast}$.
Therefore, we conclude that $\rho_{\text{ex}}(\widetilde \Phi_\epsilon,K_\epsilon)$ is a strictly increasing function as $\epsilon\to \epsilon_0$.  
This combining with continuity of the function shows item (ii) of Theorem~\ref{parametrised}.

Now, for any $\epsilon\in(0,\epsilon_0]$, we can choose 
some parameter $\epsilon'<\epsilon$, such that 
there exists some periodic orbit 
$p'\in J_{\epsilon'}$ with positive rational rotation number. 
Thus, 
$\widetilde{\rho}_{\text{ex}}(\widetilde \Phi_\epsilon,K_\epsilon)\geq 
\widetilde{\rho}_{\text{ex}}(\widetilde \Phi_{\epsilon'},K_{\epsilon'})>0$ 
for the proper lift. 
So $\rho_{\text{ex}}(\Phi_\epsilon,K_\epsilon)\neq 0 \text{ mod } \Bbb Z$ for $\epsilon\in (0,\epsilon_0]$, 
and thus the $K_\epsilon$ is a rotational attractor. 

We next argue that attractors $\{K_{\epsilon}\}_{\epsilon\in I}$ 
are indeed Lakes of Wada continua. 
Denote by $U_{\text{out}}$ the complementary domain of 
$G_2$ bounded by $C_2$ and by $U_0$ and $U_1$
 the complementary domains bounded by the circles $C_0$ and $C_1$, respectively. 
Suppose $\underbar x\in \underleftarrow{\lim}\{D_2, \psi_\epsilon\}$.
If for $\tau \in\{\text{out}, 0, 1\}$, and for a positive integer $m$, 
we have $\pi_{-m}(\underbar x)\in U_\tau$, 
then it can be verified by the definition of $\psi_\epsilon$ in (\ref{definition_of_psi}), that, 
for all $k\geq m$, 
$\pi_{-k} (\underbar x) \in U_\tau$.

Thus, for $\tau \in\{\text{out}, 0, 1\}$,
 and $m\geq 0$, one defines the following. 
\begin{align}
U_{\tau,m,\epsilon} 
&: =\{\underbar{x}\in\underleftarrow{\lim}\{D_2, \psi_\epsilon\} \big| m \text{ is the least  integer such that }
\pi_{-m}(\underbar{x})\in U_{\tau} \}.
\end{align}
\begin{align}
{L}_{\tau,\epsilon} & :=\bigcup_{m=1}^{\infty}{U}_{\tau,m,\epsilon}.
\end{align}
Clearly, 
\begin{equation}
\underleftarrow{\lim}\{D_2, \psi_\epsilon\} 
= K_\epsilon\sqcup L_{\text{out},\epsilon} \sqcup L_{0,\epsilon}\sqcup L_{1,\epsilon},
\end{equation}
where  
$L_{\text{out},\epsilon}, L_{0,\epsilon},
L_{1,\epsilon}$ 
 are three disjoint $\Phi_\epsilon$-invariant domains. 
Now, for any 
$\underbar y=(\ldots,y_{-1},y_0) \in K_\epsilon$, 
choose arbitrarily $\tau\in\{\text{out},0,1\}$, and integer $k\geq 1$. 
We pick  
$\underbar x\in U_{\tau,k+1,\epsilon}$, 
such that, for all 
$j\leq k$, $\pi_{-j}(\underbar x)=\pi_{-j}(\underbar y)$, 
and $\pi_{-(k+1)}(\underbar x)\in U_\tau$. 
By the arbitrary choice of $k\geq 1$, 
and by the topology of $\underleftarrow{\lim}\{D_2, \psi_\epsilon\}$, 
this shows that
for any small neighbourhood $B$ of $\underbar y$ in $\underleftarrow{\lim}\{D_2, \psi_\epsilon\}$, there is some point
$\underbar x\in L_{\tau,\epsilon}\cap B$.
In other words, 
$\underbar y$ is a boundary point of the domain $L_{\tau,\epsilon}$. Thus $K_\epsilon$ is the common boundary of each of 
the domains $L_{\text{out},\epsilon},L_{0,\epsilon}$ and $L_{1,\epsilon}$.
The proof of item (iii) is thus complete.

Now we address (iv). For any $\delta>0$, 
we can choose sufficiently small $\epsilon$, with $d_{C^0}(\phi_\epsilon,\text{id})<\delta.$
Thus, by the construction of both $\Upsilon_\epsilon$ and $\overline{\phi}_\epsilon$ from 
Step 2, by reducing $\epsilon$ one more time if necessary, we can ensure 
that for all sufficiently small 
$\epsilon>0$, $d_{C^0}( \Phi_\epsilon,\text{id})\leq \delta,$ which proves item (iv).

To address (v), note that for every $\epsilon\in I$,
the topological entropy of $\phi_\epsilon$ is $\log \frac{1}{1-2\epsilon}$, since $\phi_{\epsilon}$ is a piecewise monotone map with constant slope $\pm \lambda=\pm \frac{1}{1-2\epsilon}$ (see~\cite{Misiurewicz_entropy}). 
This is also equal to $h_\text{top}(\Phi_\epsilon \big|_{K_\epsilon})$, since the natural extension of an endomorphism
has the same 
topological entropy, see \cite{Li}.
In particular, 
$h_\text{top}(\Phi_\epsilon \big|_{K_\epsilon})$ decreases and converges to $0$ as 
$\epsilon \to 0$.

Lastly, we check (vi). 
Choose a point $z_0$ contained in the intersection of $U_{\text{out}}$
with $\psi_\epsilon^{-1}(G_2)$, 
and choose some pre-image $z_{-1}\in \psi_\epsilon^{-1}(z_0)$. We can find 
an arc  $\lambda$ connecting $z_{-1}$ to $z_0$, which is contained in $U_{\text{out}}$. 
Inductively, for any $k\geq 1$, 
we can define an arc
$\lambda_{-(k+1)}\subset U_{\text{out}}$ connecting some point $z_{-(k+1)}$ and $z_{-k}$, 
such that $\psi_\epsilon(\lambda_{-(k+1)}) =  \lambda_{-k}$. 
Note that for all $k\geq 1$, $z_{-k}\in \psi_{\epsilon}^{-k}(z_0)$. 
 Now we consider 
the inverse limit set
$\underline{\lambda}:= (\cdots, \lambda_{-2}, \lambda_{-1},\lambda)$. 
Moreover, we define 
$\underline\Gamma$ as follows.
\begin{equation}
\underline \Gamma:= 
\bigcup_{n \in  Z} \Phi^n_\epsilon(\underline{\lambda})\subset L_{\text{out},\epsilon}.
\end{equation} 
From the above argument, it is clear that 
$\underline \Gamma$ 
is a translation line for $\Phi_\epsilon$.

Clearly, 
$\omega(\underline \Gamma)= K_\epsilon$.
Then, the filled set $\bar K_\epsilon$ is a 
rotational attractor, which is disjoint from $\underline \Gamma$. 
This shows item (vi).

We finally remark that
the cases for $n>3$ can be dealt with analogously. 
The only difference is that the maps 
$\phi_{\epsilon,n}: G_n\to G_n$ where $G_n=S_0 \cup S_1 \cup\ldots S_{n-1}$ 
are defined separately for the upper and lower halves 
of the circles $S_i$ where $i\in\{1,\ldots, n-2\}$ 
(see Figure~\ref{fig_three_circles}). We omit the details.
\end{proof}

\begin{remark}\label{nonexpansive}
Recall that a homeomorphism $h: \mathbb{R}^2\to  \mathbb{R}^2$ 
is \emph{expansive} if for some $c > 0$,
$\sup_{n\in \Bbb Z} \text{dist}(h^n(x),h^n(x))>c$ for any pair 
$x, y\in  \mathbb{R}^2$.
Here
$\Phi_\epsilon$ restricted to $K_\epsilon$ is not expansive.
We can choose two points $x\neq x'\in S_0$ 
so that $\phi_{\epsilon}(x)=\phi_{\epsilon}(x')\in S_1$. 
Note that such points exist for every $\epsilon>0$ and they can be chosen so that their distance on $S_0$ is arbitrarily small. 
Furthermore, we can choose 
$\mathrm{dist}(\phi^{-i}_{\epsilon}(x)-\phi^{-i}_{\epsilon}(x'))\to 0$ 
as $i\to \infty$. Therefore, the induced shift homeomorphism is not expansive. Similar argument works for the cases when $n>3$ as well.
\end{remark}

\section{Wada Lakes Rotational Attractors without fixed points}\label{with_no_fixedpoints}
In this section, we 
prove Theorem~\ref{thm:noFP}, which provides 
a family of 
Lakes of Wada rotational attractors
without fixed points in the boundary of the attractor. 

\begin{proof}[Proof of Theorem \ref{thm:noFP}]
Set $n=2$, so we aim to 
construct $4$-separating Lakes of Wada rotational attractor with no fixed points in its boundary. 
	Denote by $D_3$ a closed  topological disk minus a union of three disjoint open disks. 
	Let $G_3=S_0\cup S_1\cup S_2 \subset D_3 \subset \mathbb{S}^2$ 
	be a chain of three circles. Suppose 
	$S_0$ and $S_1$ intersect
	 at a single point $a_1$.
	Let $S_1$ and $S_2$ intersect at 
	a single point $a_3$.
	Denote points $q_4\in S_0$, $q_1, q_3\in S_1$ and $q_2\in S_2$ 
	as depicted in Figure~\ref{fig_three_circles}.
	
	\begin{figure}[!ht]
		\begin{center}			\includegraphics[width=11cm]{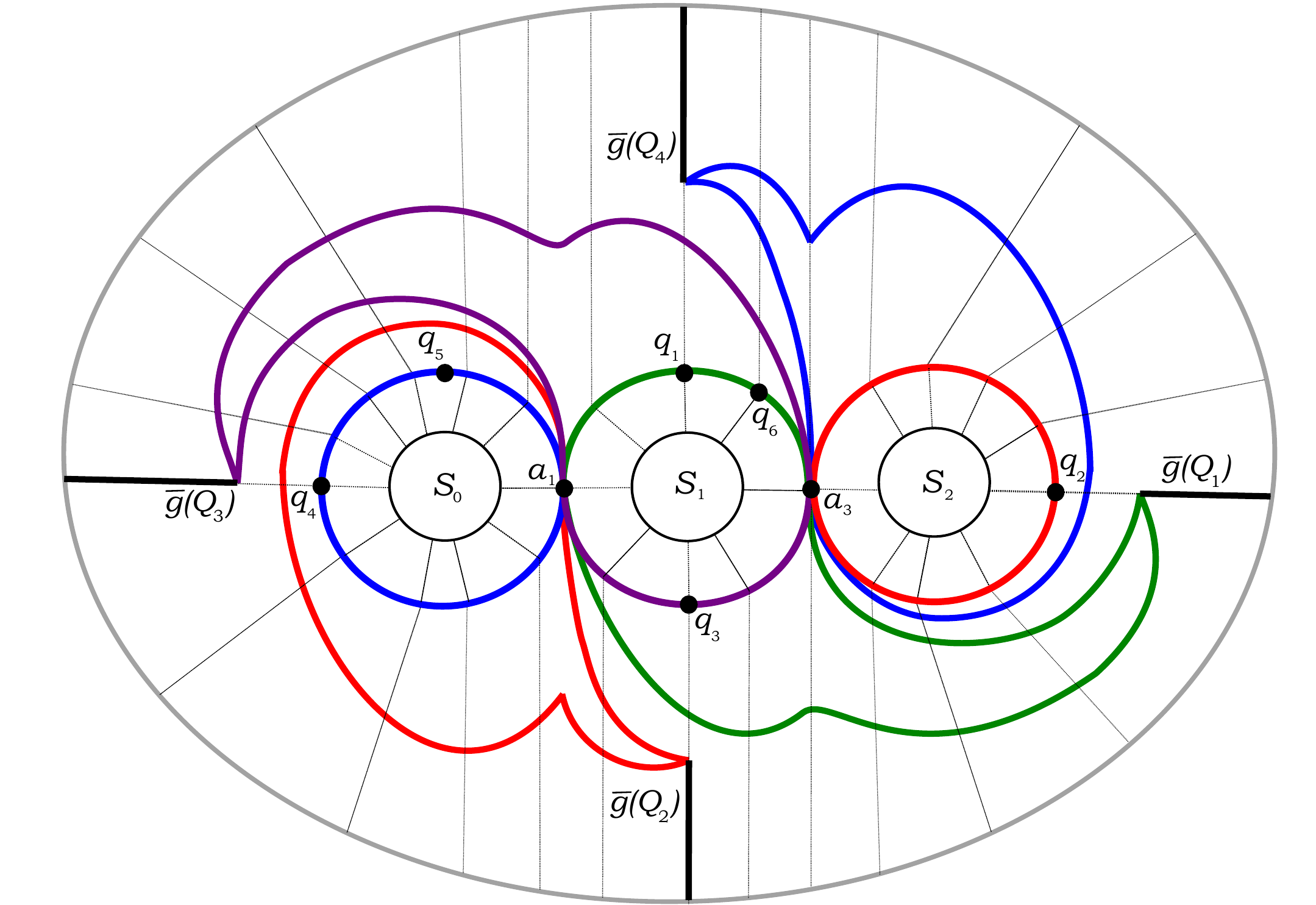}
			\caption{The unwrapping of the map $g$ and decomposition of $D_3$ and representations of the first iteration of arcs $Q_1$,$Q_2$, $Q_3$ and $Q_4$ with $\overline{g}$.} 
				~\label{fig_three_circles}
		\end{center}
	\end{figure}

	Similar to the definition of the map $\phi_{\epsilon}$ in the proof of Theorem~\ref{parametrised}, 
	we define 
	a piecewise uniform scaling map $\xi:G_3\to G_3$, satisfying the following conditions (see Figure~\ref{fig_three_circles}).
	\begin{enumerate} 
	\item Restricted to $S_0$, $\xi$ maps the arc from $a_1$ to $q_5$ 
	to the whole circle $S_0$, uniformly scaled and counterclockwise;
	it maps the arc from $q_5$ to $q_4$ to the the arc from $a_1$ to $q_3$ in a uniform scaling way; 
	it maps the arc from 
	$q_4$ to $a_1$ to the arc from $q_3$ to $a_1$, reversing the orientation.
	\item Restricted to the upper half circle of $S_1$, 
	$\xi$ maps the arc from $a_1$ to $q_1$ 
	to the upper arc from $a_1$ to $q_4$ in a uniform scaling way. 
	It maps the arc from $q_1$ to $q_6$ to the upper arc from $q_4$ to $a_1$, uniformly scaled. 
	It maps
	the arc from $q_6$ to $a_3$ to the upper arc from $a_1$ to $a_3$, uniformly scaled.
	\item Finally, restricted to the lower half circle of $S_1$ and restricted to $S_2$, the definition of $\xi$ 
	is given in a symmetric way. 
	We omit the precise description and refer to  Figure~\ref{fig_three_circles}.
	\end{enumerate}

	Let $R:G_3\to G_3$ be the rotation of $G_3$  by the angle 
	$\pi$ around the central point of the circle $S_1$. 
	Define $g: =R\circ \xi: G_3\to G_3$. 
	Observe that, 
	there exists a $4$-periodic orbit of $g$, 
	 namely, $\{q_1, q_2, q_3 , q_4\}$.
	
	 Recall that we denote 
	 by $\sigma_g$ the corresponding shift homeomorphism
	 for the inverse limit space $\underleftarrow{\mathrm{lim}}\{G_3 ,g\}$.
	 Similar to the case in 
	 the previous section, 
	 $G_3$ is a boundary retract of $D_3$. 
	  We define the smash mapping $\Upsilon$ and 
	  the unwrapping $\overline g$ in a much simpler way
	  than  in Section~\ref{lakes_wada}, in particular, here we only need to define one embedding, instead of 
	  a parametrised family and thus we just follow a construction from \cite{BM};
	  see Figure~\ref{fig_three_circles}. Now we can apply 
	 Lemma~\ref{BBM}, to obtain an extension 
	 $\psi: D_3\to D_3$ of 
	 the unwrapping $\overline g$, 
	 as well as a homeomorphism 
	 $\Phi=\sigma_{\psi}$ 
	 of $D_3$ (which is homeomorphic to $\underleftarrow{\mathrm{lim}}\{D_3,\psi\}$),
	  for which 
	 $K=\underleftarrow{\mathrm{lim}}\{G_3 ,g\}$
	 is an attractor.
	 $K$ is homeomorphic to a
	 Lakes of Wada continuum. 
	 The argument to show this follows exactly the lines of
	 the proof of item (iii) of Theorem~\ref{parametrised}. 
	 
	What remains to be checked 
	for Theorem~\ref{thm:noFP} is that $K$ 
	has no fixed points and that its 
	external prime ends rotation number is non-zero.

	\begin{claim}\label{claimnoFP}
		$\Phi |_K=\sigma_g$ has no fixed points. 
	\end{claim}
	\begin{proof}[Proof of Claim \ref{claimnoFP}]
	To prove this it is enough to show that the map $g$ has no fixed points. 
	But this is obvious from the definition of $g$. 
	Therefore the induced homeomorphism $\sigma_g$ has no fixed points by \cite{Li} as well.
	\end{proof}
	
	\begin{claim}\label{claimRotation2}
		The external prime ends rotation number $\rho_{\text{ex}}(\Phi,K) \neq 0 \text{ mod } \Bbb Z$.
	\end{claim}
	\begin{proof}[Proof of Claim \ref{claimRotation2}]
	Note that we have already specified a 
	$4$-periodic orbit of $g$ being $\{q_1,q_2,q_2,q_4\}\subset G_3$. 
	Then we obtain elements in the inverse limit space, namely, 
	\begin{align}
	\underbar{q}_1:=(\ldots,q_2,q_3,q_4,q_1).\\
	\underbar{q}_1:=(\ldots,q_3,q_4,q_1,q_2).\\
	\underbar{q}_3:=(\ldots,q_4,q_1,q_2,q_3).\\
	\underbar{q}_4:=(\ldots,q_1,q_2,q_3,q_4). 
	\end{align} 
	These form a $4$-periodic orbit for $\Phi$ 
	in 
	$\underline{\lim}\{G_3,g\}$.
	We proceed as in the proof of Lemma~\ref{claim1}.
	By attaching arcs, we can show that 
	these four points are indeed 
	all accessible points of 
	the continuum $K$. Clearly, again by Lemma~\ref{Luis_lemma},
	the dynamics over this periodic orbit of accessible points 
	imply that the external prime ends rotation number $\rho_{\text{ex}}(\Phi,K)$
	is $\frac 34 \text{ mod } \Bbb Z$, which is nonzero. This concludes the proof of Claim \ref{claimRotation2}.
	\end{proof}
	The part of the statement about translation line 
	$\Gamma$ is argued in the same way as in the proof of 
	Theorem~\ref{parametrised}. Arguments concerning $n>2$ follow analogously as described above
	for $n=2$, working on $G_{2n-1}$ being the chain of $(2n-1)$-circles. 
\end{proof}

\section{Cofrontier Dynamics and Embeddings}\label{applications}
We include two more applications of the BBM technique in this section, 
by proving Theorem~\ref{pseudo-circles} and 
Theorem~\ref{NoEndPoint}.
\subsection{Conjecture of Walker revisited}\label{Walker}
As we have already mentioned in the introduction, in the spirit of Walker's conjecture, it is still a question if 
certain cofrontier dynamics can induce two different irrational prime ends rotation numbers 
in the two complementary domains. Walker's paper \cite{Walker} did not contain results with such properties. In this section, we answer this question affirmatively by proving 
Theorem~\ref{pseudo-circles}.
\begin{proof}[Proof of Theorem~\ref{pseudo-circles}]
Let us begin by defining a 
parametrised family
$\{ \widetilde {g}_t\}_{t \geq 0} \subset \widetilde {\text{End}}_1(\mathbb{S}^1)$
 as follows. 
 \begin{equation}
 \widetilde g_t(\widetilde x)=\widetilde x+\frac{t}{2 \pi}\sin(2\pi \widetilde x).
 \end{equation}
This is a reduced Arnold's family (see \cite{Boyland_1986_bifurcations} for more).
 It is known that when $t\geq t_\ast$ for some $t_\ast > 1$, the rotation set
$\rho(\widetilde g_t)$ is a nondegenerate interval. 
Note that this family consists of odd functions. 
It follows immediately that,
$\rho(\widetilde g_t)$ is of the form $[-\eta_t,\eta_t]$ for some $\eta_t>0$.
 On the other hand, 
these rotation intervals  change continuously (Lemma 3.1 of \cite{Boyland_1986_bifurcations}).
Clearly, we can choose $I_\ast=[t_\ast,t_{\ast}']$, such that $\eta_{t_{\ast}'}>\eta_{t_\ast}$.

In what follows, we will work with $D_1$ which is the closed annulus. The round circle $G_1=\Bbb S^1$
can be regarded as the spine of $D_1$.

The main idea is to again apply Lemma~\ref{BBM} to study appropriate inverse limit spaces. 
Denote the outer boundary and the inner boundary of $D_1$ by $C_1$ and $C_0$, respectively. 
Therefore, we obtain the radial decomposition $\alpha: C_0\bigcup C_1\times [0,1] \to D_1$ as before. 
Note that, for all $t\in I_{\ast}$, the bonding map $g_t$ is a circle endomorphism with two turns. 

Define the smash mapping $\Upsilon_t$
as in \cite{BCH}.
Then we define the unwrapping
$\overline g_t$, similar to 
the definition of 
$\overline{\phi}_\epsilon$ in
Definition~\ref{overlinephi_definition}; see Figure~\ref{unwrapping_definition}. More precisely, the definition of the unwrapping $\overline g_t$ is given as follows (c.f. Definition~\ref{overlinephi_definition}).

\begin{enumerate}
	\item $\overline g_t$ is a near-homeomorphism. $\overline g_t$ restricted to $G_1$ is a homeomorphism onto its image.
	\item $\Upsilon_t\circ \overline g_t\big|_{G_1}= g_t$.
	\item there exists an interval $J_{\text{out},t} \subset G_1$, which corresponds to 
	the efficient climbing interval of $g_t$, such that if for any $y \in J_{\text{out},t}$ we
	denote $x\in C_1$ such that $\alpha(x,1) =y$, then 
	$\overline g_t$ restricted to the radial arc $\alpha(x,[0,1])$
	is a monotone map, whose image is contained in the radial arc connecting $g_t(y)$ to 
	the corresponding point in $C_1$.
	\item there exists an interval $J_{\text{in},t}\subset G_1$, which corresponds to 
	the lower climbing interval of $g_t$, 
	such that if for any $y'\in J_{\text{in},t}$ we denote $x'\in C_0$ with $\alpha(x',1) =y'$, then
	$\overline g_t$ restricted to the radial arc $\alpha(x',[0,1])$
	is a monotone map, whose image is contained in the radial arc connecting 
	$g_t(y')$ to the corresponding point in $C_0$.
\end{enumerate}

The full details are left to the reader because of analogy with the proof of Theorem~\ref{parametrised}.
We proceed to 
define the homeomorphism $\psi_t$ which extends $g_t$ to $D_1$, as well as the shift homeomorphism
$\Phi_t=\sigma_{\psi_t}$, 
exactly the same way as in Step 2 of the proof of Theorem~\ref{parametrised}. 
By Lemma~\ref{BBM}, for each $t$, 
we obtain the inverse limit space 
$\underleftarrow{\lim}\{D_1, \psi_t\}$, 
which is homeomorphic to $D_1$.
Moreover, the inverse limit space $\underleftarrow{\lim}\{G_1,g_t\}$
is homeomorphic to $K_t$. 

Now, each
inverse limit space $K_t$
is a circle-like continuum by definition. 
It follows that every proper 
subcontinuum of $K_t$ is 
chainable and thus non-separating. 
Since clearly $K_t$ has empty interior, it follows 
$K_t$ is a cofrontier with two complementary domains for each $t\in I_{\ast}$. To proceed we need the following claim.

\begin{claim}\label{Accessibility_for_Walker}
Suppose that
$\{\Phi_t\}_{t\in I_\ast}$ and
$\{K_t\}_{t\in I_\ast}$ are as above. 
Then for any $t \in I_\ast$ the following properties hold.
\begin{enumerate}[(a)]
\item There exists an accessible point $\underbar p_{\text{ex}}\in K_t$ (respectively, 
$\underbar p_{\text{in}} \in K_t$), corresponding to some point $p_{\text{ex}}\in G_1$ (respectively, $p_{\text{in}}\in G_1$),
whose  forward iterates (respectively, backward iterates) rotation number
is equal to the upper endpoint of the rotation interval, 
$\sup \rho(\widetilde {g}_t)$
(respectively, to the lower endpoint $\inf \rho(\widetilde{g}_t)$). 
\item For a certain lift $\widetilde \Phi_t$, 
the lifted external prime ends rotation number $\widetilde{\rho}_{\text{ex}}(\widetilde \Phi_t,K_t)$ equals 
$\sup \rho(\widetilde {g}_t)$ 
and lifted internal prime ends rotation number $\widetilde{\rho}_{\text{in}}(\widetilde \Phi_t,K_t)$ equals 
$\inf \rho(\widetilde {g}_t)$. 
\item The attractor 
$K_t$ is an indecomposable continuum.
\end{enumerate}
\end{claim}
	\begin{remark}The proof is a simplified 
	version of a part of the proof of 
	Theorem~\ref{parametrised}, with a few variations. 
	So we only sketch the main points and stress the differences.
	\end{remark}
\begin{proof}[Proof of Claim~\ref{Accessibility_for_Walker}]
In item (a), for the efficient climbing interval  $J_{\text{out},t}$, we can apply Lemma~\ref{one_turn} to find some point $p_{\text{ex}}\in J_{\text{out},t}$, 
and choices of 
its backward iterates, $\{p_{-\ell}\}_{\ell\geq 1}\subset J_{\text{out},t}$, such that 
the backward rotation number of $p_{\text{ex}}$ and the forward rotation 
number of $p_{\text{ex}}$ coincide and equal to the upper endpoint of 
the rotation interval $\eta_t=\sup \rho(\widetilde g_t)$.
Then we denote $\underbar p:=(\ldots,p_{-2},p_{-1},p_{\text{ex}})\in K_t \simeq \underleftarrow{\lim}(G_1, g_t)$.

Now, similar to what we did in the proof of Theorem~\ref{parametrised},
one can denote $Q_0$ for the radial arc connecting $p_{\text{ex}}$ to some point in 
the outer boundary $C_1$ 
of $D_1$, and define $Q_{-\ell}$ to be the radial subarc of the arc connecting $p_{-\ell}$ to the corresponding 
point in $C_1$ so that $\psi_t(Q_{-\ell})=Q_{-(\ell-1)}$ for all $\ell\geq 1$.
It follows that the inverse limit set
$\underbar Q:= (\ldots,Q_{-1},Q_0)\subset \underleftarrow {\lim}\{D_1, \psi_t\}$
is an arc, and $\underbar Q\bigcap K_t=\{\underbar p\}$. 
Therefore, $\underbar p$ is an accessible point of $K_t$
from the exterior domain by the arc $\underbar Q$. 
This completes half of the proof of item (a) of Claim~\ref{Accessibility_for_Walker}. 
For the other part, we use the definition of lower climbing interval. 
In a similar way, one shows that there is point $p_{\text{in}}\in G_1$, whose forward rotation number and the backward rotation number coincide and equal $-\eta_t=\inf \rho(\widetilde g_t)$. 
Then we can obtain $\underbar p_{\text{in}}$ in $K_t$, which is accessible by an arc 
from interior of the annulus. Item (a) is thus complete.

Now we check item (b). 
Observe that, 
by the choice of the point $p_{\text{ex}}$, 
the forward and backward rotation numbers of $p_{\text{ex}}$ coincide. 
Then by Lemma~\ref{Luis_lemma}, 
this number is equal to 
the lifted exterior prime ends rotation number 
$\widetilde{\rho}_{\text{ex}}(\widetilde \Phi_t,K_t)$, for a proper lift. 
Then by item (a), $\widetilde{\rho}_{\text{ex}}(\widetilde \Phi_t,K_t)$ is equal to $\eta_t$. 
The situation for the
interior prime ends rotation number is similar.

For item (c), recall  that the rotation set $\rho(\widetilde g_t)$ is a non-degenerate segment.
By Theorem 2.7 of~\cite{BG_periodicity}, 
$K_t$ is an indecomposable continuum for every $t\in I_{\ast}$. 
\end{proof}

We are now ready to finish the proof of Theorem~\ref{pseudo-circles}.
By  
Claim~\ref{Accessibility_for_Walker}, for $t\in I_\ast$, 
the rotation set 
$\rho(\widetilde \Phi_t\big|_{K_t})$ is a non-trivial segment. 
By definition, item (i) is proved, namely $K_t$ are indeed Birkhoff-like attractors. 
 
To show item (ii), note that for any $t\in I_\ast$, 
the circle map $g_t$ is topologically exact. Then by
Theorem 22 in~\cite{KOT}, 
for any $\epsilon>0$, 
there is another map $g_t'$, 
with $d_{C^0}(g_t,g'_t)<\epsilon$, 
such that the
inverse limit $\underleftarrow{\lim}\{G_1,g_t'\}$ 
is the pseudo-circle (see Theorem 3.2 of~\cite{BO} where this argument was applied as well). 
Therefore, we can apply a similar 
construction as we did in Theorem~\ref{parametrised} with a properly chosen 
smash function and the unwrapping. Therefore, 
the invariant pseudo-circle Birkhoff-like attractor $K'_t$ can be embedded 
such that $d_H(K'_t, K_t)< \epsilon$.
 This shows item (ii).
  
To show item (iii), 
observe that when $t$ varies in $I_\ast$ 
the rotation set $\rho(\widetilde g_t)$
changes continuously in a strictly monotone way from $[-\eta_{t_\ast},\eta_{t_\ast}]$ to $[-\eta_{t_\ast'},\eta_{t_\ast'}]$, with $\eta_{t_\ast'}>\eta_{t_\ast}$.
It follows that, there are uncountably many parameters $t\in I_\ast$ for which $\eta_t\notin \Bbb Q$.
In particular, for those choices of $t$, the two lifted prime ends rotation numbers are different and both are irrationals. 
\end{proof}

In the proof of this theorem, we thus already showed the following statement.
\begin{corollary} For a circle endomorphism with two turns, we can embed the inverse limit space such that,
the exterior prime ends rotation number and the interior prime ends rotation number equal the 
two endpoints of the rotation interval of the shift homeomorphism restricted to it. 
\end{corollary}

\subsection{Prime ends rotation number realising 
an interior point of the rotation interval}

In this subsection, we prove 
Theorem~\ref{NoEndPoint}. The proof is much simpler and does not require parametrised family of
inverse limit embeddings. 
In fact, we only give one circle endomorphism and we will exhibit two different embeddings 
by drawing the graph of the unwrapping. 
	\begin{figure}[!ht]
		\begin{center}
			\includegraphics[width=12cm]{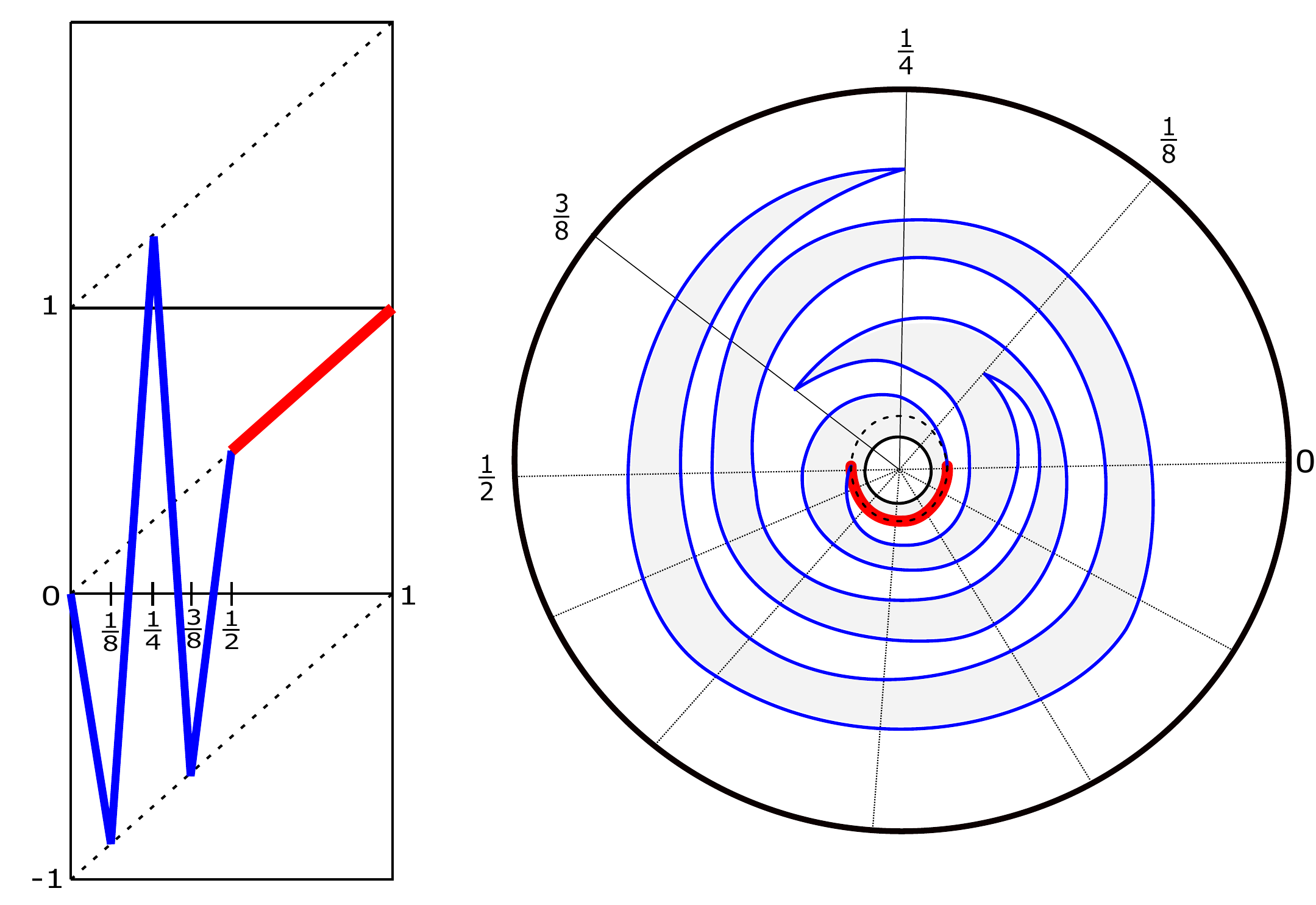}
			\caption{On the left, it is the graph of a lift $\widetilde g$ of the bonding map $g$, 
			 whose rotation set $\rho(\widetilde g)$ is $[-1,1]$. 
			 On the right, it is the graph of the unwrapping of $g$, 
			which eventually gives an arc of fixed point in the confontier, that are accessible from interior.} ~\label{fig_prime_neq_rotation}
		\end{center}
	\end{figure}
\begin{proof}[Proof of Theorem~\ref{NoEndPoint}] 
Similar to the arguments we have seen in previous sections, 
we construct 
the attractor as an inverse limit of a continuous circle endomorphism $g$. 
However, for our purpose here, the circle map $g$ is not an endomorphism with two turns. 
By definition, $g$ is affine restricted to the  five consecutive 
subintervals 
splitting the circle (see left of Figure~\ref{fig_prime_neq_rotation}).  
Then we define the unwrapping $\overline g$ as depicted on the 
right of 
Figure~\ref{fig_prime_neq_rotation}. 
We omit the precise definitions because they are much simpler 
than what we did in previous sections. 
Then, like in the proof of Theorem \ref{thm:noFP}, we can define the smash mapping $\Upsilon$ preserving radial segments, 
as well as the 
near-homeomorphism $\psi$ which extends $g$ to $D_1$.

Then we invoke Lemma~\ref{BBM}, to obtain 
the inverse limit space $\underleftarrow{\lim}\{D_1,\psi\}$, which 
is homeomorphic to $D_1$, 
and obtain 
the shift homeomorphism $\Phi=\sigma_\psi$ 
on $\underleftarrow{\lim}\{D_1,\psi\}$,
with 
the attractor $K=\underleftarrow{\lim}\{ \Bbb S^1,g\}$, where $\Bbb S^1$ is the circle.
 As we have already argued in 
the previous subsection, 
$K$ is a cofrontier.

Note that the lift $\widetilde g$ at point $1/8$ and $3/8$
has rotation number equal to $-1$, and at point $1/4$ has rotation number 
$1$. So the rotation interval $\rho(\widetilde g)$ is non-trivial (in fact it is equal to $[-1,1]$ by more detailed analysis of the map $g$ which we omit here).  
Moreover, by the relation (4.1) in~\cite{BCH},
we have that $\rho(\widetilde \Phi,K)=\rho(\widetilde g)=[-1,1]$ 
for proper choice of lifts. 

On the other hand, note that 
$\widetilde g$ point-wise fixes the arc $[1/2,1]$, 
where $\widetilde g$ has rotation number $0$. 
Applying similar arguments as in the proof of Lemma~\ref{claim1},
we  show that there is a semi-circle of fixed points that are accessible from interior. 
By Lemma~\ref{Luis_lemma} we conclude that the lifted interior prime ends rotation 
number $\rho_{\text{in}}(\widetilde{\Phi},K)$ equals $0$. Similar considerations show that the 
lifted exterior prime ends rotation number $\rho_{\text{ex}}(\widetilde{\Phi},K)$ equals $1$. 

It remains to be shown that the shift homeomorphism on $K$ (i.e., 
the restriction of $\Phi$ on $K$)
is not transitive. 
This follows from the fact that 
$g$ is not transitive, because it contains an arc of fixed points, 
and from the fact that 
$\omega((\ldots,x_2,x_1),\sigma_g)=\underleftarrow{\mathrm{lim}}\{\omega(x_1,g),g\}$ 
for any $x_1\in \mathbb{S}^1$ (see for example \cite{Li}). 

For the last assertion of the theorem, 
we will consider a different embedding of 
$\underleftarrow{\mathrm{lim}}(\mathbb{S}^1,g)$
 that makes the arc of fixed points inaccessible.
 The idea is to define the unwrapping in a different way. 
 Let us skip the precise definition 
 and only refer to Figure~\ref{fig_inaccessible} below for the graph of the unwrapping. 
 Then following the same argument invoking Lemma~\ref{BBM}, 
 one obtains the cofrontier attractor $K'$, which 
 is homeomorphic to $K$ by construction. 
 The proof of Theorem~\ref{NoEndPoint} is now completed. 
 \end{proof}
 \begin{figure}[!ht]
		\begin{center}
			\includegraphics[width=12cm]{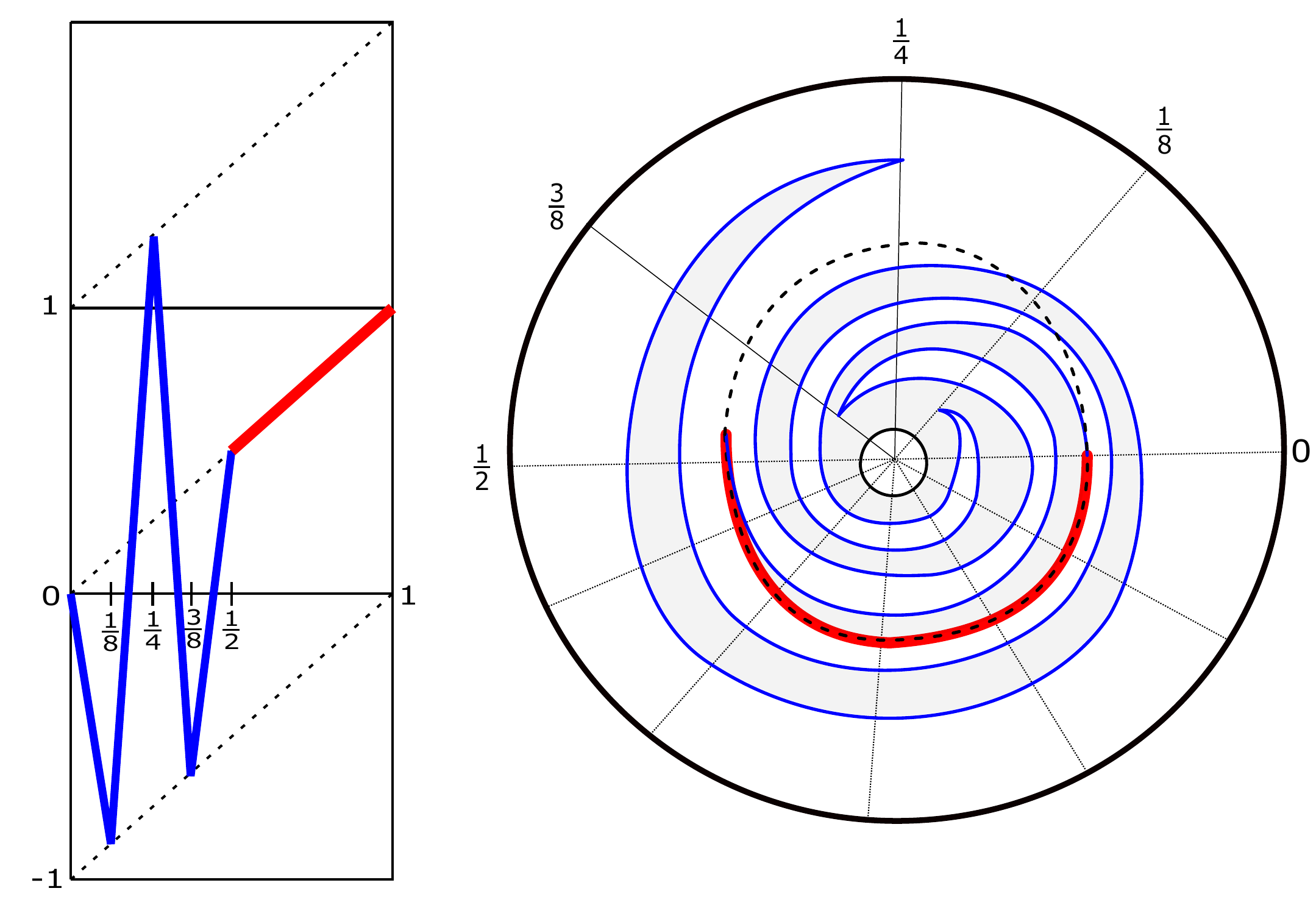}
			\caption{With the same bonding map on the left, 
			on the right, we draw the graph of a different unwrapping. 
			This provides a BBM embedding which has an arc of fixed points that are inaccessible.} ~\label{fig_inaccessible}
		\end{center}
	\end{figure}
\begin{remark}
	Note that we could (similarly as in previous sections) obtain a parametrised family of rotational attractors by modifying the map $\widetilde{g}$ and move three critical points in $(0,1/2)$ parallely with $y-axis$ uniformly towards the map $\widetilde{id}$. Embedding the spaces similarly as it is suggested on Figure~\ref{fig_prime_neq_rotation} we would obtain a parametrised family of examples answering the question of Boyland above (of course excluding the map $\widetilde{id}$). Choosing an appropriate smash similarly as in Theorem~\ref{smash} we can even obtain a parametrised family of such non-transitive examples arbitrary close to identity. 
\end{remark}

\begin{remark}
Two BBM embeddings are called equivalent if there is 
a conjugacy $h$ which restricted to the attractor is identity. 
 In Theorem 2.17 from \cite{Boy}, (based on previous results from \cite{Oversteegen_extend})
  the authors proved 
that any two BBM embeddings 
of unimodal inverse limit spaces are equivalent. In Theorem~\ref{NoEndPoint} we provide a contrast picture for circle-like 
continua as demonstrated with the two non-equivalent planar embeddings obtained from the BBM construction. 
\end{remark}

\section{Acknowledgements}
The first two authors are supported by University of Ostrava subsidy for institutional development IRP201824 “Complex topological
structures” and the NPU II project LQ1602 IT4 Innovations excellence in science. JB is grateful to Luis Hern\'andez-Corbato, Jose M.R. Sanjurjo and Francisco R. Ruiz del Portal for some useful conversations and hospitality during author's visit at Universidad Complutense de Madrid in October 2018. J\v C was also supported by FWF Schrödinger Fellowship stand-alone project J-4276.
X-C. Liu is supported by 
Fapesp P\'os-Doutorado grant (Grant Number
 2018/03762-2). 
 X-C. Liu thanks Ana Anu\v si\'c for useful conversations. We thank also P. Boyland, A. Koropecki and A. Passeggi for useful remarks, as well as Salvador Addas-Zanata and Fabio Armando Tal
 for valuable comments on Subsection 6.1, and for suggesting to consider reduced Arnold's family. 
 
\bibliographystyle{plain}
\bibliography{Rotational_attractor17}
\end{document}